\documentclass[11pt]{article}
\usepackage{amsfonts}
\usepackage{graphicx}
\usepackage{epstopdf}
\usepackage{algorithmic}
\usepackage{subfig}
\usepackage{amssymb}
\usepackage{amsmath}
\usepackage{amsthm}
\usepackage{accents}

\newcommand{\mybold}[1]{\mathbf{\boldsymbol{#1}}}
\renewcommand{\vec}[1]{\underline{\mybold{#1}}}

\newcommand{\lvec}[1]{\accentset{\text{\tiny$\leftarrow$}}{#1}}
\newcommand{\rvec}[1]{\accentset{\text{\tiny$\rightarrow$}}{#1}}
\newcommand{\lrvec}[1]{\accentset{\text{\tiny$\leftrightarrow$}}{#1}}

\newcommand{\tvec}[1]{\rvec{\mybold{#1}}}
\newcommand{\tmat}[1]{\rvec{\mybold{#1}}}
\newcommand{\ivec}[1]{\lvec{\mybold{#1}}}
\newcommand{\imat}[1]{\lvec{\mybold{#1}}}
\newcommand{\itvec}[1]{\lrvec{\mybold{#1}}}

\newcommand{\xscal}[1]{\mybold{\mathsf{#1}}}
\newcommand{\xvec}[1]{\vec{\mathsf{#1}}}

\newcommand{\iter}[3][i]{\vec{#2}_{#1}^{(#3)}}
\newcommand{\itersc}[3][i]{#2_{#1}^{(#3)}}
\newcommand{\iterscal}[2]{\xscal{#1}^{(#2)}}
\newcommand{\itervec}[2]{\xvec{#1}^{(#2)}}

\newcommand{\dx}{\Delta x}
\newcommand{\dy}{\Delta y}
\newcommand{\dt}{\Delta t}
\newcommand{\dtau}{\Delta \tau}
\newcommand{\kron}{\otimes}
\newcommand{\had}{\circ}
\newcommand{\diag}[1]{\mathrm{diag}(#1)}
\newcommand{\vspan}[1]{\mathrm{span}(#1)}

\newtheorem{theorem}{Theorem}

\newtheorem{lemma}[theorem]{Lemma}
\newtheorem{definition}{Definition}
\newtheorem{assumption}{Assumption}
\newtheorem{remark}{Remark}

\title{Locally conservative and flux consistent iterative methods}
\author{Viktor Linders$^{\mbox{\tiny\rm 1}}$, Philipp Birken$^{\mbox{\tiny\rm 1}}$}
\date{}

\begin{document}
\maketitle
\baselineskip=0.9
\normalbaselineskip
\vspace{-3pt}
\begin{center}{\footnotesize\em $^{\mbox{\tiny\rm 1}}$Centre for
    mathematical sciences, Lund University, Lund, Sweden.\\ email: viktor.linders@math.lu.se \\ \qquad philipp.birken@na.lu.se \\ }
\end{center}

\begin{abstract}
Conservation and consistency are fundamental properties of discretizations of systems of hyperbolic conservation laws. Here, these concepts are extended to the realm of iterative methods by formally defining \emph{locally conservative} and \emph{flux consistent} iterations. These concepts are of both theoretical and practical importance: Based on recent work by the authors, it is shown that pseudo-time iterations using explicit Runge-Kutta methods are locally conservative but not necessarily flux consistent. An extension of the Lax-Wendroff theorem is presented, revealing convergence towards weak solutions of a temporally retarded system of conservation laws. Each equation is modified in the same way, namely by a particular scalar factor multiplying the spatial flux terms. A technique for enforcing flux consistency, and thereby recovering convergence, is presented. Further, local conservation is established for all Krylov subspace methods, with and without restarts, and for Newton's method under certain assumptions on the discretization. Thus it is shown that Newton-Krylov methods are locally conservative, although not necessarily flux consistent. Numerical experiments with the 2D compressible Euler equations corroborate the theoretical results. Further numerical investigations of the impact of flux consistency on Newton-Krylov methods indicate that its effect is case dependent, and diminishes as the number of iterations grow.
\end{abstract}


%========================================================================

\section{Introduction}

Conservation laws arise ubiquitously in the modelling of physical phenomena and their discretizations remain the subject of intense research. Fundamental properties of successful schemes include conservation, consistency and convergence. These concepts are well defined for both space and time discretizations; explicit and implicit.

Implicit discretizations typically result in a large, sparse systems of nonlinear equations to be solved in each time step. The solution is usually approximated through the application of iterative methods; see e.g. \cite[Chapters 5 \& 6]{birken2021numerical}. Yet, discussions about conservation, consistency and convergence (in the sense of grid refinement) for schemes involving iterative methods are rare. In \cite{jespersen1983flux,barth1987analysis}, studys were conducted of particular implicit finite volume schemes applied to the steady Euler equations, solved using a variety of modified Newton-type methods. The results indicate that the choice of iterative method has a significant impact on the convergence of the scheme. Based on these results, a study of similar schemes applied to the Reynolds-Averaged Navier-Stokes (RANS) equations were carried out in \cite{junqueira2014study}, where it was found that the less performant methods violate mass conservation.

In \cite{birken2021conservative}, the authors considered general finite volume discretizations of 1D scalar conservation laws, discretized in time with the implicit Euler method. Global (i.e. mass) conservation was proven for many methods, including pseudo-time iterations, Krylov subspace methods, Newton's method and certain multigrid techniques. On the other hand, the Jacobi and Gauss-Seidel iterations were shown to violate mass conservation in general, corroborating the observations in \cite{junqueira2014study}. A stronger notion is that of {\it local conservation}, defined formally below, which loosely means that mass is not only conserved but also not teleported. This notion is important for physical correcteness, and allows to prove extensions of the Lax-Wendroff theorem \cite{lax1959systems}, thus giving a much stronger mathematical backing of such nonlinear schemes. In \cite{birken2021conservative}, a start was made in this vein for pseudo-time iterations. It was shown that in case of convergence, the resulting scheme converges to a solution of a conservation law, where the flux is multiplied by a scheme dependent factor, unless particular care is taken. We say that the iterative method lacks {\it flux consistence}, which manifests as a temporal retardation. 

Throughout this article, we work with systems of conservation laws. After introducing relevant notation and terminology, we formally define \emph{locally conservative} and \emph{flux consistent} iterative methods in \ref{sec:preliminaries}. We extend the results on pseudotime iterations from \cite{birken2021conservative} in \ref{sec:pseudo-time}, while considering a large class of implicit Runge-Kutta (RK) methods in place of Euler's method. As it turns out, even for systems flux inconsistency manifests through a scheme dependent scalar factor. Thus, a method to enforce flux consistency, first introduced in \cite{birken2021conservative}, applies here too. 

The second focus of this work is local conservation of the important class of Newton-Krylov methods. In \ref{sec:Newton}, we first prove that Newton's method is both locally conservative and flux consistent under certain assumptions on the spatial discretization and when solving all linear systems exactly. We can thus establish that if there are problems with conservation within an implicit solver using Newton's method, they stem from the iterative solver for the linear systems. Secondly, by relating Krylov subspace methods to pseudo-time iterations, local conservation is shown also for these in \ref{sec:krylov}. This subsequently leads to a proof of local conservation for Newton-Krylov methods. 

Numerical examples corroborate the theoretical findings in \ref{sec:numerical_experiments}. We further explore the impact of flux consistency on Newton-Krylov methods by applying the aforementioned technique for ensuring flux consistency of pseudo-time iterations. The results indicate that role of flux consistency is case dependent, and that its effect diminishes as the number of iterations grow.

% Precise details are provided in \ref{sec:preliminaries}, where definitions of conservative and consistent iterative methods are also introduced. The conservation and consistency theory for pseudo-time iterations is extended to the new setting in \ref{sec:pseudo-time}. In \ref{sec:krylov}, Krylov subspace methods are shown to be related to pseudo-time iterations. Local conservation is established for these methods, although the question of consistency is left unanswered. Under certain assumptions on the finite volume discretization, Newton's method is shown to be both conservative and consistent in \ref{sec:Newton}. This leads to the conclusion that Newton-Krylov methods are locally conservative, although consistency is once again unclear. Numerical experiments in \ref{sec:numerical_experiments} corroborate the theoretical findings. Further, it is shown experimentally that Newton-Krylov methods can be accelerated by manually enforcing consistency, essentially at no additional cost. Conclusions are drawn in \ref{sec:conclusions}.

%=========================================================================

\section{Preliminaries} \label{sec:preliminaries}

This section introduces relevant notation and the theoretical background upon which the remaining paper rests.

\subsection{Notation}

Scalar quantities are denoted by letters in normal font. Vectors and matrices are bold, with vectors being lower case and matrices upper case. Vectors and matrices of several different dimensions are treated in the manuscript:
\begin{itemize}
\item We consider systems of $m$ conservation laws. Vectors in $\mathbb{R}^m$ are written with an underline e.g. $\vec{u}$.
\item We consider $s$-stage explicit Runge-Kutta methods. Vectors in $\mathbb{R}^s$ are written with a right-pointing arrow, e.g. $\tvec{b}$, and similarly for matrices in $\mathbb{R}^{s \times s}$.
\item We consider $\tilde{s}$-stage implicit Runge-Kutta methods. Vectors in $\mathbb{R}^{\tilde{s}}$ are written with a left-pointing arrow e.g. $\ivec{b}$, and similarly for matrices in $\mathbb{R}^{\tilde{s} \times \tilde{s}}$.
\item Vectors whose dimension is the product of the dimensions listed above are written with a combination of attributes, e.g. $\itvec{x} \in \mathbb{R}^{\tilde{s}s}$, $\vec{\tvec{y}} \in \mathbb{R}^{sm}$, $\vec{\itvec{z}} \in \mathbb{R}^{\tilde{s}sm}$.
\item Vectors representing quantities on spatial grids are expressed in a bold sans serif font, e.g.
$$
\xscal{u}^\top = (\dots, u_{i-1},u_i,u_{i+1}, \dots).
$$
This notation is combined with the accents above if the evaluated quantity in question is a vector. Thus, a vector of $m$-element vectors is represented as
$$
\xvec{u}^\top = (\dots,\vec{u}_{i-1}^\top,\vec{u}_i^\top,\vec{u}_{i+1}^\top, \dots).
$$
Matrices operating on these vectors are denoted similarly with capital letters and are constructed block-diagonally:
$$
\xvec{A} = 
%\dots \oplus \vec{A}_{i-1} \oplus \vec{A}_i \oplus \vec{A}_{i+1} \oplus \dots = 
\text{blkdiag}(\dots, \vec{A}_{i-1}, \vec{A}_i, \vec{A}_{i+1}, \dots).
$$
\item A flux function that takes $(p+q+1)$ arguments, e.g. $\vec{f}_{i+\frac{1}{2}}(\vec{u}_{i-p}, \dots, \vec{u}_{i+q})$, is sometimes denoted with the abbreviated argument $\vec{f}_{i+\frac{1}{2}}(\xvec{u})$.
\end{itemize}

%=========================================================================

\subsection{Conservation laws and the Lax-Wendroff theorem}

Consider the system of 1D conservation laws
\begin{equation} \label{eq:conservation_law}
\vec{u}_t + \vec{f}_x = \vec{0}, \quad \vec{u}(x,0) = \vec{u}_0(x), \quad x \in \Omega, \quad t > 0.
\end{equation}
Here, $\vec{u}(x,t), \, \vec{f}, \, \vec{u}_0 \in \mathbb{R}^m$. Throughout, it is assumed that \eqref{eq:conservation_law} is posed either as a Cauchy problem or on a periodic domain. Under these circumstances, the quantity $\int_\Omega \vec{u} \text{d}x$ is conserved. A space-time discretization of \eqref{eq:conservation_law} that discretely mimics this property is said to be \emph{globally conservative}.

We consider discretizations of \eqref{eq:conservation_law} that may be expressed in the form
\begin{equation} \label{eq:FV}
\frac{\vec{u}_i^{n+1} - \vec{u}_i^n}{\dt} + \frac{1}{\dx} \left( \vec{f}_{i + \frac{1}{2}} - \vec{f}_{i - \frac{1}{2}} \right) = \vec{0}, \quad i = \dots, -1, 0, 1, \dots
\end{equation}
Here, $\vec{u}_i^n \approx \vec{u}(x_i,t_n) \equiv \vec{u}(i \dx, n \dt)$. In this paper we consider implicit discretizations and therefore restrict our attention to numerical fluxes of the type
$$
\vec{f}_{i + \frac{1}{2}} \equiv \vec{f}_{i + \frac{1}{2}}(\xvec{u}^{n+1}) = \vec{f}_{i + \frac{1}{2}}(\vec{u}_{i-p}^{n+1}, \dots, \vec{u}_{i+q}^{n+1}),
$$
where $p$ and $q$ are nonnegative integers with $p+q > 0$. Throughout, it is assumed that these numerical fluxes are consistent:

% DEFINITION: CONSISTENT NUMERICAL FLUX
\begin{definition} \label{def:consistent_flux}
The numerical flux $\vec{f}_{i + \frac{1}{2}}(\xvec{u}) = \vec{f}_{i + \frac{1}{2}}(\vec{u}_{i-p}, \dots, \vec{u}_{i+q})$ is said to be consistent with $\vec{f}(\vec{u})$ if it it Lipschitz continuous in each argument and if $\vec{f}_{i + \frac{1}{2}}(\vec{u},\dots,\vec{u}) = \vec{f}(\vec{u})$.
\end{definition}
% END OF DEFINITION

The concept of local conservation will be central in the remainder. It applies to explicit and implicit discretizations alike:

% DEFINITION: LOCAL CONSERVATION
\begin{definition}
A discretization of \eqref{eq:conservation_law} that can be expressed in the form \eqref{eq:FV} is said to be locally conservative.
\end{definition}
% END OF DEFINITION

Local conservation is a useful property for both physical and mathematical reasons \eqref{eq:conservation_law}. It enforces that solution components leaving one computational cell necessarily enter the neighbouring one. Conservation of "total mass", $\sum_i \Delta x \vec{u}_i^n$ is thereby ensured, in analogy with the continuous problem (i.e. global conservation). Further, it is an essential ingredient in the ubiquitous Lax-Wendroff theorem.

The Lax-Wendroff theorem applies to the Cauchy problem for \eqref{eq:conservation_law} and considers locally conservative discretizations with consistent numerical flux. If the numerical solution of such a scheme converges to a function $\vec{u}$ in the limit of vanishing $\dx$ and $\dt$, the theorem provides sufficient conditions for $\vec{u}$ to be a weak solution of the conservation law \eqref{eq:conservation_law} \cite[Chapter 12]{leveque1992numerical}. More precisely, consider a sequence of grids $(\dx_\ell, \dt_\ell)$ such that $\dx_\ell, \dt_\ell \rightarrow 0$ as $\ell \rightarrow \infty$. Let $\vec{\mathcal{U}}_\ell(x,t)$ denote the piecewise constant function that takes the solution value $\vec{u}_i^n$ in $(x_i, x_{i+1}] \times (t_{n-1}, t_n]$ on the $\ell$th grid. We make the following assumptions:

% ASSUMPTIONS
\begin{assumption} \label{assumptions} 
\hfill
\begin{enumerate}
    \item There is a function $\vec{u}(x,t)$ such that over every bounded set $\Omega = [a,b] \times [0,T]$ in $x$-$t$ space,
    $$
    \| \vec{\mathcal{U}}_\ell(x,t) - \vec{u}(x,t) \|_{1,\Omega} \rightarrow 0 \quad \text{as} \quad \ell \rightarrow \infty.
    $$
    \item For each $T \geq 0$ there is a constant $R > 0$ such that the total variation
    $$
    TV(\vec{\mathcal{U}}_\ell(\cdot,t)) < R \quad \text{for all} \quad 0 \leq t \leq T, \quad \ell = 1, 2, \dots
    $$
\end{enumerate}
\end{assumption}
% END OF ASSUMPTIONS

The Lax-Wendroff theorem can then be stated as follows:

% THEOREM: LAX-WENDROFF
\begin{theorem} \label{thm:LW_original}
Consider a sequence of grids $(\dx_\ell, \dt_\ell)$ such that $\dx_\ell, \dt_\ell \rightarrow 0$ as $\ell \rightarrow \infty$. Suppose that the numerical flux $\vec{f}_{i \pm \frac{1}{2}}$ in \eqref{eq:FV} is consistent with $\vec{f}$ and that \ref{assumptions} is satisfied. Then, $\vec{u}(x,t)$ is a weak solution of \eqref{eq:conservation_law}.
\end{theorem}
% END OF THEOREM

\ref{assumptions} is not strictly speaking necessary for the Lax-Wendroff theorem. Both conditions can be relaxed somewhat. Indeed, the original proof due to Lax and Wendroff instead assumes that $\vec{\mathcal{{U}}}_\ell$ converges boundedly almost everywhere to $\vec{u}$ \cite{lax1959systems}.

The discretization \eqref{eq:FV} is implicit, hence the solution generally must be approximated using iterative methods. Since we are interested in the convergence properties of schemes involving iterative methods, the concepts of local conservation and consistency must be extended to this setting. To this end, we propose the following definition:

% DEFINITION: CONSERVATIVE AND CONSISTENT ITERATIVE METHODS
\begin{definition} \label{def:conservative_iterative_method}
Suppose that the solution to \eqref{eq:FV} is approximated by a sequence of iterates $\iter{u}{k}, \, k=0,\dots,N$ and set $\vec{u}^{n+1} = \iter{u}{N}$. If there is a numerical flux function $\vec{h}_{i+\frac{1}{2}}^{(N)}$ such that the approximate numerical solution $\vec{u}^{n+1}$ satisfies
\begin{equation} \label{eq:conservative_form}
    \frac{\vec{u}_i^{n+1} - \vec{u}_i^n}{\dt} + \frac{1}{\dx} \left( \vec{h}_{i+\frac{1}{2}}^{(N)} - \vec{h}_{i-\frac{1}{2}}^{(N)} \right) = \vec{0}, \qquad i = \dots, -1, 0, 1, \dots
\end{equation}
then the iterative method is said to be locally conservative. If $\vec{h}_{i+\frac{1}{2}}^{(N)}$ further satisfies \ref{def:consistent_flux}, then the iterative method is said to be flux consistent.
\end{definition}
% END OF DEFINITION

%=========================================================================

\section{Conservation and consistency of pseudo-time iterations} \label{sec:pseudo-time}

In order to approximate the solution of the nonlinear system \eqref{eq:FV} using pseudo-time iterations, we introduce a pseudo-time derivative,
\begin{equation*}
	\frac{\partial \vec{u}_i}{\partial \tau} + \vec{g}_i(\xvec{u}) = 0, \qquad \vec{u}_i(0) = \vec{u}_{0_i}, \qquad i = \dots, -1, 0, 1, \dots
\end{equation*}
where the nonlinear function $\vec{g}_i$ is given by
\begin{equation} \label{eq:g-fun}
    \vec{g}_i(\xvec{u}) = \frac{\vec{u}_i - \vec{u}_i^n}{\dt} + \frac{1}{\dx} \left( \vec{f}_{i+\frac{1}{2}}(\xvec{u}) - \vec{f}_{i-\frac{1}{2}}(\xvec{u}) \right).
\end{equation}
Several different methods are available for iterating in pseudo-time \cite{swanson2007convergence, birken2019preconditioned}. Herein, we use an $s$-stage explicit Runge-Kutta (ERK) method. Let $(\tmat{A},\tvec{b},\tvec{c})$ denote the coefficient matrix and vectors of the ERK method. We denote the $k$th pseudo-time iterate by $\iter{u}{k}$. The subsequent iterate $\iter{u}{k+1}$ is computed from $\iter{u}{k}$ as
\begin{equation} \label{eq:RK_step}
    \iter{u}{k+1} = \iter{u}{k} - \dtau_k \sum_{j=1}^s b_j \vec{g}_i \left( \xvec{U}_j^{(k)} \right), \qquad i = \dots, -1, 0, 1, \dots
\end{equation}
where the stage vectors $\xvec{U}_j^{(k)}$, $j = 1, \dots, s$ have elements
\begin{equation} \label{eq:RK_stages}
\vec{U}_{j_\iota}^{(k)} = \iter[\iota]{u}{k} - \dtau_k \sum_{l=1}^{j-1} a_{j,l} \vec{g}_\iota \left( \xvec{U}_l^{(k)} \right), \quad \iota = i-p, \dots, i+q.
\end{equation}
As previously, $p$ and $q$ determine the bandwidth of the finite volume stencil.

%========================================================================

\subsection{Systems of conservation laws}

In \cite{birken2021conservative} the scalar version (i.e. $m=1$) of \eqref{eq:conservation_law} was considered. There, it was shown that the scheme \eqref{eq:RK_step}--\eqref{eq:RK_stages} preserves the local conservation of the space-time discretization. Further, an extension of the Lax-Wendroff theorem was provided that incorporates a fixed number $N$ of pseudo-time iterations.

Here, we present two theorems that generalize the results in \cite{birken2021conservative} to systems of conservation laws (i.e. $m \geq 1$). We define a step in physical time by setting $\vec{u}_i^{n+1} = \iter{u}{N}$. Throughout, $\iter{u}{0} = \vec{u}_i^n$ is chosen as initial guess.

Recall that the \emph{stability function} $\phi(z)$ of an RK method $(\tmat{A},\tvec{b},\tvec{c})$ is given by
\begin{equation} \label{eq:stability_function}
    \phi(z) = 1 + z \tvec{b}^\top (\tmat{I} - z \tmat{A})^{-1} \tvec{1},
\end{equation}
where $\tmat{I}$ is the $s \times s$ identity matrix and $\tvec{1} \in \mathbb{R}^s$ is the vector of all ones; see e.g. \cite[Chapter IV.3]{wanner1996solving}. The \emph{stability region} of the RK method is defined as the subset of the complex plane for which $|\phi(z)| < 1$.

The proofs of the following theorems are very similar to those presented for the scalar case in \cite{birken2021conservative}. The details are therefore omitted. The following is a generalization to systems of conservation laws, which also allows each pseudo-time step to be taken with different ERK methods.

% THEOREM: LOCAL CONSERVATION
\begin{theorem} \label{thm:local_conservation}
Choose the initial guess $\iter{u}{0} = \vec{u}_i^n$. Apply $N$ pseudo-time iterations to \eqref{eq:FV}, where the $k$th iteration is performed with an ERK method $(\tmat{A}_k,\tvec{b}_k,\tvec{c}_k)$ with stability function $\phi_k(z)$ and pseudo-time step $\dtau_k$. Let $\mu_k = \dtau_k/\dt$ for $k = 0, \dots, N-1$. The pseudo-time iterations are locally conservative with numerical flux
\begin{equation} \label{eq:H-flux}
    \vec{h}_{i+\frac{1}{2}}^{(N)} = \sum_{k=0}^{N-1} \left( \mu_k \tvec{b}_k^\top (\tmat{I} + \mu_k \tmat{A}_k)^{-1} \kron \vec{I} \right) \left( \prod_{l=k+1}^{N-1} \phi_l(-\mu_{l}) \right) \vec{\tvec{f}}_{i+\frac{1}{2}}^{(k)}.
\end{equation}
Here, $\vec{I}$ is the $m \times m$ identity matrix and
$$
\vec{\tvec{f}}_{i+\frac{1}{2}}^{(k)} = \left( \vec{f}^\top_{i+\frac{1}{2}}\left( \xvec{U}_1^{(k)} \right), \dots, \vec{f}^\top_{i+\frac{1}{2}} \left( \xvec{U}_s^{(k)} \right) \right)^\top \in \mathbb{R}^{sm}.
$$
The numerical flux is consistent with $c \vec{f}(u)$, where
\begin{equation} \label{eq:wave_speed}
    c \equiv c(\mu_0,\dots,\mu_{N-1}) = 1 - \prod_{l=0}^{N-1} \phi_l(-\mu_{l}).
\end{equation}
Thus, pseudo-time iterations are flux consistent if and only if $c=1$.
\end{theorem}
% END OF THEOREM

% REMARK: PRODUCT DEFINITION
\begin{remark}
The product in \eqref{eq:H-flux} is empty when $k = N-1$. To handle this case we use the convention
$$
\prod_{l=N}^{N-1} \phi_l(-\mu_{l}) = 1.
$$
\end{remark}
% END OF REMARK

% PROOF
\begin{proof}
The proof is step by step the same as those of \cite[Lemma 3 \& Theorem 2]{birken2021conservative}. The only changes necessary are to replace $\tmat{A}$ by $\tmat{A} \kron \vec{I}$, and to swap each multiplication of the form $\tmat{1} \phi$, where $\phi$ is a scalar, to a corresponding Kronecker product $\tmat{1} \kron \vec{\phi}$.
\end{proof}
% END OF PROOF

\ref{thm:local_conservation} reveals that pseudo-time iterations are locally conservative but not necessarily flux consistent, as characterized by the scalar $c$ in \eqref{eq:wave_speed}. To remove the inconsistency, it suffices to make a single iteration with any explicit RK method for which $\phi_l(-\mu_l) = 0$. In \cite{birken2021conservative} it was suggested to iterate once with the explicit Euler method, choosing $\dtau_0 = \dt$ so that $\mu_0 = 1$. The stability function is $\phi(-\mu_0) = 1-\mu_0$, hence $\phi(-1) = 0$. This rids the iterative method of its flux inconsistency and the remaining iterations can be made with any other ERK method as preferred. Numerical experiments in \cite{birken2021conservative} showed that enforcing flux consistency can have a profound impact on the convergence of the pseudo-time iterations.

The coefficient $c$ causes the numerical fluxes in \eqref{eq:H-flux} to be consistent with the modified system of conservation laws
\begin{equation} \label{eq:pseudo_conservation_law}
    \vec{u}_t + c(\mu_0,\dots,\mu_{N-1}) \vec{f}_x = 0.
\end{equation}
Note that $c$ affects all components of the system identically. An interpretation of its presence is that the pseudo-time iterations alter the rate of flow of time. Defining the modified time $t_c = ct$ it follows that $\vec{u}_t = c \vec{u}_{t_c}$ so that $\vec{u}_{t_c} + \vec{f}_x = \vec{0}$. Hence, $t_c$ rather than $t$ is the governing time variable in \eqref{eq:pseudo_conservation_law}.

Convergent pseudo-time iterations require that the pseudo-time steps are restricted to the stability domain of the explicit RK method, at least for most of the iterations. Thus, the product in \eqref{eq:wave_speed} is generally taken over factors bounded by unity, and consequently $c \leq 1$. Thus, $t_c$ represents a time retardation, i.e. time flows slower in \eqref{eq:pseudo_conservation_law} than in the original conservation law \eqref{eq:conservation_law}. This is in line with the experimental observations made in \cite{birken2021conservative}.

%========================================================================

\subsection{Higher order implicit Runge-Kutta methods}

Let us return to the system of conservation laws \eqref{eq:conservation_law} and introduce a spatial semi-discretization
\begin{equation} \label{eq:semi-discretization}
	\vec{u}_t + \frac{1}{\dx} \left( \vec{f}_{i+\frac{1}{2}}(\xvec{u}) - \vec{f}_{i-\frac{1}{2}}(\xvec{u}) \right) = \vec{0}.
\end{equation}
If the implicit Euler method is used to discretize \eqref{eq:semi-discretization} in time, then \eqref{eq:FV} is recovered. However, suppose that we instead wish to discretize in time using an $\tilde{s}$-stage implicit Runge-Kutta (IRK) method with Butcher matrix and vectors $(\imat{A},\ivec{b},\ivec{c})$. The resulting scheme can be expressed as
\begin{equation} \label{eq:IRK}
\begin{aligned}
	\vec{U}_i^{(j)} &= \vec{u}_i^n - \frac{\dt}{\dx} \sum_{l=1}^{\tilde{s}} \lvec{a}_{jl} \left( \vec{f}_{i+\frac{1}{2}}(\xvec{U}^{(j)}) - \vec{f}_{i-\frac{1}{2}}(\xvec{U}^{(j)}) \right), \quad j = 1, \dots, \tilde{s}, \\
	%%%
	\vec{u}_i^{n+1} &= \vec{u}_i^n - \frac{\dt}{\dx} \sum_{j=1}^{\tilde{s}} \lvec{b}_j \left( \vec{f}_{i+\frac{1}{2}}(\xvec{U}^{(j)}) - \vec{f}_{i-\frac{1}{2}}(\xvec{U}^{(j)}) \right).
\end{aligned}
\end{equation}
Define the quantities
\begin{align*}
	\vec{\ivec{U}}_i &= \left( \vec{U}_i^{(1),\top}, \dots, \vec{U}_i^{(\tilde{s}),\top} \right)^\top \in \mathbb{R}^{\tilde{s}m}, \\
	\vec{\ivec{f}}_{i+\frac{1}{2}}(\ivec{\xvec{U}}) &= \left( \vec{f}_{i+\frac{1}{2}}^\top(\xvec{U}^{(1)}), \dots, \vec{f}_{i+\frac{1}{2}}^\top(\xvec{U}^{(\tilde{s})}) \right)^\top \in \mathbb{R}^{\tilde{s}m}.
\end{align*}
Then, \eqref{eq:IRK} can be reformulated as
\begin{equation} \label{eq:IRK_system}
\begin{aligned}
	\vec{\ivec{U}}_i &= \ivec{1} \kron \vec{u}_i^n - \frac{\dt}{\dx} \left( \imat{A} \kron \vec{I} \right) \left( \vec{\ivec{f}}_{i+\frac{1}{2}}(\ivec{\xvec{U}}) - \vec{\ivec{f}}_{i-\frac{1}{2}}(\ivec{\xvec{U}}) \right), \\
	%%%
	\vec{u}_i^{n+1} &= \vec{u}_i^n - \frac{\dt}{\dx} \left( \ivec{b}^\top \kron \vec{I} \right) \left( \vec{\ivec{f}}_{i+\frac{1}{2}}(\ivec{\xvec{U}}) - \vec{\ivec{f}}_{i-\frac{1}{2}}(\ivec{\xvec{U}}) \right),
\end{aligned}
\end{equation}
where $\ivec{1}$ is the vector of ones in $\mathbb{R}^{\tilde{s}}$.

We seek an approximation of the solution $\vec{u}_i^{n+1}$. To this end we must find an approximate solution of the nonlinear equation system in the second line of \eqref{eq:IRK_system}. Rewriting this equation as
\begin{equation} \label{eq:IRK_conservative_stages}
	\frac{\vec{\ivec{U}}_i - \ivec{1} \kron \vec{u}_i^n}{\dt} + \frac{1}{\dx} \left( \imat{A} \kron \vec{I} \right) \left( \vec{\ivec{f}}_{i+\frac{1}{2}}(\ivec{\xvec{U}}) - \vec{\ivec{f}}_{i-\frac{1}{2}}(\ivec{\xvec{U}}) \right) = \vec{\ivec{0}},
\end{equation}
we see that it is of precisely the same form as the discretization \eqref{eq:FV}, although with $\vec{\ivec{U}}_i$ taking the place of $\vec{u}_i^{n+1}$, $\lvec{1} \kron \vec{u}_i^n$ replacing $\vec{u}_i^n$ and $(\imat{A} \kron \vec{I}) \vec{\ivec{f}}_{i+\frac{1}{2}}$ in place of $\vec{f}_{i+\frac{1}{2}}$. Thus, if pseudo-time iterations are used to approximate a solution, then \ref{thm:local_conservation} applies and we can immediately conclude that the resulting scheme can be written in the conservative form
\begin{equation} \label{eq:IRK_stage_flux_form}
	\frac{\vec{\ivec{U}}_i - \ivec{1} \kron \vec{u}_i^n}{\dt} + \frac{1}{\dx} \left( \vec{\ivec{h}}_{i+\frac{1}{2}}^{(N)} - \vec{\ivec{h}}_{i-\frac{1}{2}}^{(N)} \right) = \vec{\ivec{0}},
\end{equation}
where the numerical flux is given by
$$
\vec{\ivec{h}}_{i+\frac{1}{2}}^{(N)} = \sum_{k=0}^{N-1} \left( \mu_k \tvec{b}^\top (\tmat{I} + \mu_k \tmat{A}) \kron \ivec{A} \kron \vec{I} \right) \left( \prod_{l=k+1}^{N-1} \phi_l(\mu_l) \right) \vec{\lrvec{f}}_{i+\frac{1}{2}}^{(k)},
$$
and
$$
\vec{\lrvec{f}}_{i+\frac{1}{2}}^{(k)} = \left( \vec{\lvec{f}}^\top \left( \xvec{\lvec{U}}_1^{(k)} \right), \dots, \vec{\lvec{f}}^\top \left( \xvec{\lvec{U}}_s^{(k)} \right) \right)^\top.
$$
At this point we note that if the scheme is such that a vector $\ivec{v}$ exists, satisfying
\begin{equation} \label{eq:v_conditions}
	\ivec{v}^\top \ivec{A} = \ivec{b}^\top, \qquad \ivec{v}^\top \ivec{1} = 1,
\end{equation}
then the second line in \eqref{eq:IRK_system} can be evaluated by left-multiplying the first line by $\ivec{v}^\top \kron \vec{I}$. In other words, it follows that
\begin{equation} \label{eq:IRK_step_definition}
	\vec{u}_i^{n+1} = \left( \ivec{v}^\top \kron \vec{I} \right) \vec{\ivec{U}}_i.
\end{equation}
If we adopt this principle and use it to compute $\vec{u}_i^{n+1}$ based on the approximation of $\vec{\ivec{U}}_i$ obtained by applying pseudo-time iterations to \eqref{eq:IRK_conservative_stages}, then the resulting scheme is conservative:

% THEOREM: RUNGE-KUTTA
\begin{theorem} \label{thm:Runge-Kutta}
Apply $N$ pseudo-time iterations to the stage equations \eqref{eq:IRK_conservative_stages} with the same assumptions as in \ref{thm:local_conservation}. Suppose that a vector $\ivec{v}$ exists that satisfies conditions \eqref{eq:v_conditions} and compute $\vec{u}_i^{n+1}$ using \eqref{eq:IRK_step_definition}. Then the pseudo-time iterations are locally conservative with the numerical flux
\begin{equation} \label{eq:IRK_numerical_flux}
	\vec{h}_{i+\frac{1}{2}}^{(N)} = \sum_{k=0}^{N-1} \left( \mu_k \tvec{b} (\tmat{I} + \mu_k \tmat{A}) \kron \ivec{b}^\top \kron \vec{I} \right) \left( \prod_{l=k+1}^{N-1} \phi(\mu_l) \right) \vec{\lrvec{f}}_{i+\frac{1}{2}}^{(k)}.
\end{equation}
The numerical flux is consistent with $c(\mu_0,\dots,\mu_{N-1}) \vec{f}$. Thus, the pseudo-time iterations are flux consistent if and only if $c=1$.
\end{theorem}
% END OF THEOREM

% PROOF
\begin{proof}
Local conservation with the flux \eqref{eq:IRK_numerical_flux} follows from left-multiplying \eqref{eq:IRK_conservative_stages} by $\left( \ivec{v}^\top \kron \vec{I} \right)$ and using \eqref{eq:v_conditions} to conclude that $\left( \ivec{v}^\top \kron \vec{I} \right)(\ivec{1} \kron \vec{u}_i^n) = \vec{u}_i^n$. Consistency with $c \vec{f}$ follows from the fact that $\ivec{b}^\top \ivec{1} = 1$ for every consistent RK method.
\end{proof}
% END OF PROOF

It should be noted that we cannot find a vector $\ivec{v}$ that satisfies conditions \eqref{eq:v_conditions} for all IRK methods. For example, the implicit midpoint rule is given by $(\ivec{A},\ivec{b},\ivec{c}) = (1,1/2,1/2)$. Thus, the first condition in \eqref{eq:v_conditions} gives $\ivec{v} = 1/2$ whereas the second one gives $\ivec{v} = 1$, both of which cannot be satisfied. However, we remark that the important class of IRK methods associated with Summation-By-Parts (SBP) methods all have such a $\ivec{v}$ by construction \cite{boom2015high,linders2020properties}. This is a broad class of methods with the ability to preserve $L^2$-type estimates of the solution to systems of differential equations \cite{nordstrom2013summation} and encompass the ubiquitous Radau IA and IIA and Lobatto IIIC methods \cite{ranocha2019some} as special cases. For further details about the theoretical and practical aspects of SBP methods for time marching, see \cite{versbach2022theoretical} and the references therein.

A generalization of the Lax-Wendroff theorem can be obtained if we restrict ourselves to considering a fixed number of iterations on a sequence of ever finer grids:

% THEOREM: LAX-WENDROFF
\begin{theorem} \label{thm:LW}
Consider a sequence of grids $(\dx_\ell, \dt_\ell)$ such that $\dx_\ell, \dt_\ell \rightarrow 0$ as $\ell \rightarrow \infty$. Fix $N$ independently of $\ell$, set $\iter{u}{0} = \vec{u}_i^n$. Apply $N$ pseudo-time iterations to the conservative discretization \eqref{eq:FV}, or to the stage equations \eqref{eq:RK_stages} followed by the RK step \eqref{eq:RK_step}. Let $\dtau_{k,\ell} / \dt_\ell = \mu_{k,\ell} = \mu_k$ be constants independent of $\ell$ for each $k=0,\dots,N-1$. Suppose that the numerical flux $\vec{f}_{i \pm \frac{1}{2}}$ in \eqref{eq:FV} is consistent with $\vec{f}$ and that \ref{assumptions} is satisfied. Then, $\vec{u}(x,t)$ is a weak solution of the conservation law \eqref{eq:pseudo_conservation_law}.
\end{theorem}
% END OF THEOREM

% PROOF
\begin{proof}
The proof is identical to that of \cite[Theorem 3]{birken2021conservative}, with the same changes as those in the proof of \ref{thm:local_conservation}.
\end{proof}
% END OF PROOF

%========================================================================

\section{Newton's method} \label{sec:Newton}

We now return to the nonlinear system \eqref{eq:FV}, or equivalently to the implicit Runge-Kutta stage equations in \eqref{eq:IRK_system} and consider Newton's method, which replaces the nonlinear system with a sequence of linear ones. The solution to each linear system can then be approximated using pseudo-time iterations, or more commonly,using a Krylov subspace method.

In this section we limit our attention to the case when the numerical flux is bivariate, i.e. when $\vec{f}_{i+\frac{1}{2}} = \hat{\vec{f}}(\vec{u}_i,\vec{u}_{i+1})$ for some function $\hat{\vec{f}}(\vec{\theta},\vec{\phi})$. It will be convenient to split the flux into a symmetric (or \emph{convective}) and an anti-symmetric (or \emph{dissipative}) component;
$$
\hat{\vec{f}} = \hat{\vec{f}}^{(+)} + \hat{\vec{f}}^{(-)}, \qquad \hat{\vec{f}}^{(+)}(\vec{\theta},\vec{\phi}) = \hat{\vec{f}}^{(+)}(\vec{\phi},\vec{\theta}), \qquad \hat{\vec{f}}^{(-)}(\vec{\theta},\vec{\phi}) = -\hat{\vec{f}}^{(-)}(\vec{\phi},\vec{\theta}).
$$
Any bivariate function can be expressed in this way by setting
$$
\hat{\vec{f}}^{(+)}(\vec{\theta},\vec{\phi}) = \frac{1}{2} (\hat{\vec{f}}(\vec{\theta},\vec{\phi}) + \hat{\vec{f}}(\vec{\phi},\vec{\theta})), \qquad \hat{\vec{f}}^{(-)}(\vec{\theta},\vec{\phi}) = \frac{1}{2} (\hat{\vec{f}}(\vec{\theta},\vec{\phi}) - \hat{\vec{f}}(\vec{\phi},\vec{\theta})).
$$
By anti-symmetry, the dissipative component satisfies $\hat{\vec{f}}^{(-)}(\vec{u},\vec{u}) = \vec{0}$. Consistency of the numerical flux is therefore equivalent to consistency of the convective component.

In the following subsections, we demonstrate that Newton's method is locally conservative and flux consistent when bivariate fluxes are used. To simplify the presentation, we begin by proving these results for scalar conservation laws. The extension to systems is straightforward but notationally complicated. We therefore postpone this to a separate subsection.

%========================================================================

\subsection{Scalar conservation laws}

Consider the scalar conservation law
\begin{equation*}
    u_t + f_x = 0,
\end{equation*}
posed on a periodic spatial domain and adjoined with appropriate initial data. Let $\hat{f}^{(\pm)}(\theta,\phi)$ be a consistent bivariate numerical flux. Without loss of generality we may assume that the flux is either symmetric or anti-symmetric. The analysis will proceed in the same way in both cases and we may thereafter form linear combinations of such fluxes as we like. We discretize with a finite volume method,
\begin{equation} \label{eq:FV_discretization}
    \frac{u_i^{n+1} - u_i^n}{\dt} + \frac{\hat{f}^{(\pm)}(u_i^{n+1}, u_{i+1}^{n+1}) - \hat{f}^{(\pm)}(u_{i-1}^{n+1}, u_i^{n+1})}{\dx} = 0.
\end{equation}
The discretization \eqref{eq:FV_discretization} may equivalently be expressed in vector form as
\begin{equation} \label{eq:FV_vector_form}
    \frac{\xscal{u}^{n+1} - \xscal{u}^n}{\dt} + \frac{1}{\dx} (\xscal{Q}^{(\mp)} \had \xscal{F}^{(\pm)}(\xscal{u}^{n+1})) \xscal{1} = \xscal{0}.
\end{equation}
The dimensions of the vectors and matrices match the number of cells in the computational grid. Here, $\had$ denotes the Hadamard product. The elements of the matrix $\xscal{F}^{(\pm)}(\xscal{u})$ are given by $(\xscal{F}^{(\pm)})_{ij} = \hat{f}^{(\pm)}(u_i,u_j)$ and $\xscal{Q}^{(\mp)}$ is given by
$$
\xscal{Q}^{(\mp)} =
\begin{bmatrix}
0 & 1 & & \dots & \mp 1 \\
\mp 1 & 0 & 1 & & \\
& \mp 1 & 0 & 1 & \\
\vdots & & & \ddots &  1\\
1 & & 0 & \mp 1 & 0
\end{bmatrix}.
$$

We define the function $\xscal{g}(\xscal{v})$ as
\begin{equation} \label{eq:Newton_vector}
    \xscal{g}(\xscal{v}) := \frac{\xscal{v} - \xscal{u}^n}{\dt} + \frac{1}{\dx} (\xscal{Q}^{(\mp)} \had \xscal{F}^{(\pm)}(\xscal{v})) \xscal{1}.
\end{equation}
Newton's method applied to the nonlinear system \eqref{eq:FV_vector_form} is then given by
\begin{equation} \label{eq:Newtons_method}
    \xscal{g}'(\iterscal{v}{k}) \Delta \xscal{v} + \xscal{g}(\iterscal{v}{k}) = \xscal{0}, \qquad \iterscal{v}{k+1} = \iterscal{v}{k} + \Delta \xscal{v},
\end{equation}
where $\xscal{g}'$ is the Jacobian of $\xscal{g}$. An explicit expression for $\xscal{g}'$ is given in \cite[Theorems 2.1 \& 5.1]{chan2022efficient}. Let $\hat{f}^{(\pm)}_\phi = \partial \hat{f}^{(\pm)} / \partial \phi$ and introduce the matrix $\xscal{F}^{(\pm)}_\phi(\xscal{v})$ with elements $(\xscal{F}^{(\pm)}_\phi)_{ij} = \hat{f}^{(\pm)}_\phi(v_i,v_j)$. Defining the matrix
$$
\xscal{\partial} \xscal{F}^{(\pm)}(\xvec{v}) := \xscal{Q}^{(\mp)} \had \xscal{F}^{(\pm)}_\phi(\xscal{v}) - \diag{\xscal{1}^\top (\xscal{Q}^{(\mp)} \had \xscal{F}^{(\pm)}_\phi(\xscal{v}))},
$$
the Jacobian is given by $\xscal{g}'(\iterscal{v}{k}) = \dt^{-1} \xscal{I} + \dx^{-1} \xscal{\partial} \xscal{F}^{(\pm)}(\iterscal{v}{k})$. This is a tridiagonal matrix, which we may explicitly write as
\begin{align*}
&\xscal{g}'(\iterscal{v}{k}) = \\
& \text{tri} \left( \mp \frac{\hat{f}^{(\pm)}_\phi(\itersc{v}{k},\itersc[i-1]{v}{k})}{\dx}, \frac{1}{\dt} - \frac{\hat{f}^{(\pm)}_\phi(\itersc[i-1]{v}{k}, \itersc{v}{k}) \mp \hat{f}^{(\pm)}_\phi(\itersc[i+1]{v}{k}, \itersc{v}{k})}{\dx}, \frac{\hat{f}^{(\pm)}_\phi(\itersc{v}{k},\itersc[i+1]{v}{k})}{\dx} \right)
\end{align*}
%
%
%\begin{equation} \label{eq:Jacobian}
%    \begin{alignedat}{3}
%    %%
%    \xscal{g}'(\iterscal{v}{k}) &= \frac{\xscal{I}}{\dt} &&+\frac{1}{\dx}
%    %
%    \begin{bmatrix}
%    & \ddots & & & \\
%    & \mp \hat{f}^{(\pm)}_\phi(\itersc{v}{k},\itersc[i-1]{v}{k}) & 0 & \hat{f}^{(\pm)}_\phi(\itersc{v}{k},\itersc[i+1]{v}{k}) & \\
%    & & & \ddots &
%    \end{bmatrix} \\
%    %
%    &&&-\frac{1}{\dx}
%    %
%    \begin{bmatrix}
%    \ddots & & \\
%    & \left( \hat{f}^{(\pm)}_\phi(\itersc[i-1]{v}{k}, \itersc{v}{k}) \mp \hat{f}^{(\pm)}_\phi(\itersc[i+1]{v}{k}, \itersc{v}{k}) \right) & \\
%    & & \ddots
%    \end{bmatrix}.
%    %
%    \end{alignedat}
%\end{equation}
%
Inserting this expression for $\xscal{g}'(\iterscal{v}{k})$ into \eqref{eq:Newtons_method} and collecting terms leads to a system of equations of the form
\begin{equation} \label{eq:Newton_locally_conservative}
    \frac{\itersc{v}{k+1} - u_i^n}{\dt} + \frac{1}{\dx} \left( h_{i+\frac{1}{2}}^{(k+1)} - h_{i-\frac{1}{2}}^{(k+1)} \right) = 0, \qquad i = \dots, -1, 0, 1, \dots,
\end{equation}
where the numerical flux function $h_{i+\frac{1}{2}}^{(k+1)}$ is given by
\begin{equation} \label{eq:Newton_flux}
    h_{i+\frac{1}{2}}^{(k+1)} = \hat{f}_\phi^{(\pm)}(\itersc{v}{k}, \itersc[i+1]{v}{k}) \Delta v_{i+1} \pm \hat{f}_\phi^{(\pm)}(\itersc[i+1]{v}{k}, \itersc{v}{k}) \Delta v_i + \hat{f}^{(\pm)}(\itersc[i+1]{v}{k}, \itersc{v}{k}).
\end{equation}
%

% THEOREM: SCALAR NEWTON
\begin{theorem} \label{thm:scalar_Newton}
Choose $\itersc{v}{0} = u_i^n$ and suppose that $\xscal{g}'(u \xscal{1})$ is nonsingular for any non-zero scalar $u$. Then Newton's method, applied to the discretization \eqref{eq:FV_vector_form}, is locally conservative and flux consistent.
\end{theorem}
% END OF THEOREM

% PROOF
\begin{proof}
Setting $u_i^{n+1} = \itersc{v}{k+1}$ in \eqref{eq:Newton_locally_conservative} shows that Newton's method applied to \eqref{eq:FV_vector_form} is locally conservative. To establish flux consistency, let $\iterscal{v}{k}$ and $\xscal{u}^n$ both be given as $u \xscal{1}$ for some nonzero scalar $u$ and some $k$ (e.g. $k=0$). Observe that, by the consistency of $\hat{f}^{(\pm)}$, we have $(\xscal{F}^{(+)}(u \xscal{1}))_{ij} = f(u)$ and $(\xscal{F}^{(-)}(u \xscal{1}))_{ij} = 0$. Consequently, $\xscal{Q}^{(+)} \had \xscal{F}^{(-)}(u\xscal{1}) = \xscal{0}$ and $\xscal{Q}^{(-)} \had \xscal{F}^{(+)}(u\xscal{1}) = f(u) \xscal{Q}^{(-)}$. Since $\xscal{Q}^{(-)} \xscal{1} = \xscal{0}$ it therefore follows from \eqref{eq:Newton_vector} that $\xscal{g}(u \xscal{1}) = \xscal{0}$. Since $\xscal{g}'(u \xscal{1})$ is nonsingular by assumption, \eqref{eq:Newtons_method} implies that $\Delta \xscal{v} = \xscal{0}$. Thus, from \eqref{eq:Newton_flux} we have
$$
h_{i+\frac{1}{2}}^{(k+1)} = \hat{f}^{(\pm)}_\phi(u,u) \cdot 0 \pm \hat{f}^{(\pm)}_\phi(u,u) \cdot 0 + \hat{f}^{(\pm)}(u,u) = \hat{f}^{(\pm)}(u,u),
$$
the latter of which is consistent by assumption. Note that $h_{i+\frac{1}{2}}^{(k+1)}$ exclusively depends on the current iterate $\itersc[\,]{v}{k+1}$ via $\Delta v$ and is therefore a linear function. The numerical flux is thus Lipschitz continuous, hence consistent.

Finally, since any bivariate function can be written as a linear combination of symmetric and anti-symmetric components, the argument above extends by linearity to arbitrary bivariate numerical fluxes.
\end{proof}
% END OF PROOF

\subsection{Systems of conservation laws}

We now return to the system of conservation laws \eqref{eq:conservation_law} and the discretization \eqref{eq:FV}, which may correspond either to the implicit Euler method or to the stage equations of a higher order RK method. The analysis proceeds very similarly to the scalar case, although we are now dealing with a flux function $\vec{f}:\mathbb{R}^m \mapsto \mathbb{R}^m$. Suppose that the corresponding numerical flux is of the form
$$
\hat{\vec{f}}^{(\pm)}(\vec{\theta},\vec{\phi}) = (\hat{f}^{(\pm)}_1(\vec{\theta},\vec{\phi}), \dots, \hat{f}^{(\pm)}_m(\vec{\theta},\vec{\phi}))^\top,
$$
where each $\hat{f}^{(\pm)}_\ell$, $\ell = 1, \dots, m$, is bivariate, consistent and either symmetric or anti-symmetric.

This time, Newton's method is given by
\begin{equation} \label{eq:Newtons_method_system}
    \xvec{g}'(\itervec{v}{k}) \Delta \xvec{v} + \xvec{g}(\itervec{v}{k}) = \xvec{0}, \qquad \itervec{v}{k+1} = \itervec{v}{k} + \Delta \xvec{v},
\end{equation}
where we define $\xvec{g}(\xvec{v})$ as
\begin{equation} \label{eq:Newton_vector_system}
\begin{aligned}
    \xvec{g}(\xvec{v}) :&= \frac{\xvec{v} - \xvec{u}^n}{\dt} + \frac{1}{\dx} ((\vec{I} \kron \xscal{Q}^{(\mp)}) \had \xvec{F}^{(\pm)}(\xvec{v})) \xvec{1} \\
    &= \frac{\xvec{v} - \xvec{u}^n}{\dt} + \frac{1}{\dx}
    \begin{bmatrix}
    (\xscal{Q}^{(\mp)} \had \xscal{F}_1) \xscal{1} \\ \vdots \\ (\xscal{Q}^{(\mp)} \had \xscal{F}_m) \xscal{1}
    \end{bmatrix}.
\end{aligned}
\end{equation}
Here we have introduced the matrix $\xvec{F}^{(\pm)} = \xscal{F}^{(\pm)}_1 \oplus \dots \oplus \xscal{F}^{(\pm)}_m$ with $(\xscal{F}^{(\pm)}_\ell)_{ij} = \hat{f}^{(\pm)}_\ell(\vec{v}_i, \vec{v}_j)$. We denote the partial derivatives of the numerical fluxes as $\hat{f}^{(\pm)}_{\ell, \phi_\kappa} = \partial \hat{f}^{(\pm)}_\ell / \partial \phi_\kappa$ and define the matrices $\xscal{F}^{(\pm)}_{\ell, \phi_{\kappa}}(\xvec{v})$ with elements $(\xscal{F}^{(\pm)}_{\ell, \phi_{\kappa}})_{ij} = \hat{f}^{(\pm)}_{\ell, \phi_\kappa}(\vec{v}_i,\vec{v}_j)$. Finally we obtain
$$
\xscal{\partial} \xscal{F}^{(\pm)}_{\ell, \phi_{\kappa}}(\xvec{v}) := \xscal{Q}^{(\mp)} \had \xscal{F}^{(\pm)}_{\ell, \phi_{\kappa}}(\xvec{v}) - \diag{ \xscal{1}^\top \left( \xscal{Q}^{(\mp)} \had \xscal{F}^{(\pm)}_{\ell, \phi_{\kappa}}(\xvec{v}) \right) }.
$$
In \cite{chan2022efficient}, the Jacobian $\xvec{g}'(\itervec{v}{k})$ is shown to be given by
$$
\xvec{g}'(\itervec{v}{k}) = \frac{\xvec{I}}{\dt} + \frac{1}{\dx}
\begin{bmatrix}
\xscal{\partial} \xscal{F}^{(\pm)}_{1, \phi_{1}}(\itervec{v}{k}) & \dots & \xscal{\partial} \xscal{F}^{(\pm)}_{1, \phi_{m}}(\itervec{v}{k}) \\
\vdots & \ddots & \vdots \\
\xscal{\partial} \xscal{F}^{(\pm)}_{m, \phi_{1}}(\itervec{v}{k}) & \dots & \xscal{\partial} \xscal{F}^{(\pm)}_{m, \phi_{m}}(\itervec{v}{k})
\end{bmatrix}.
$$
This is a block matrix formed by $m^2$ submatrices, each tridiagonal as in the scalar case (although the off-diagonal blocks do not depend on $\dt$). If we write $\Delta \xvec{v} = (\Delta \xscal{v}_1^\top, \dots, \Delta \xscal{v}_m^\top)^\top$ and collect terms, then the Newton iteration  \eqref{eq:Newtons_method_system} takes the form
\begin{align*}
\frac{\itervec{v}{k+1} - \xvec{u}^n}{\dt} &+ \frac{1}{\dx}
\begin{bmatrix}
\xscal{\partial} \xscal{F}^{(\pm)}_{1, \phi_{1}}(\itervec{v}{k}) & \dots & \xscal{\partial} \xscal{F}^{(\pm)}_{1, \phi_{m}}(\itervec{v}{k}) \\
\vdots & \ddots & \vdots \\
\xscal{\partial} \xscal{F}^{(\pm)}_{m, \phi_{1}}(\itervec{v}{k}) & \dots & \xscal{\partial} \xscal{F}^{(\pm)}_{m, \phi_{m}}(\itervec{v}{k})
\end{bmatrix}
\begin{bmatrix}
\Delta \xscal{v}_1 \\ \vdots \\ \Delta \xscal{v}_m
\end{bmatrix}
\\ &+ \frac{1}{\dx}
\begin{bmatrix}
(\xscal{Q}^{(\mp)} \had \xscal{F}^{(\pm)}_1(\itervec{v}{k})) \xscal{1} \\ \vdots \\ (\xscal{Q}^{(\mp)} \had \xscal{F}^{(\pm)}_m(\itervec{v}{k})) \xscal{1}
\end{bmatrix}
= \xvec{0}.
\end{align*}
Looking at the $i$th row of the $\ell$th block in this linear system we find that the scheme once again may be expressed on the locally conservative form
\begin{equation} \label{eq:Newton_locally_conservative_system}
    \frac{\itersc[\ell_i]{v}{k+1} - u_{\ell_i}^n}{\dt} + \frac{1}{\dx} \left( h_{\ell_i+\frac{1}{2}}^{(k+1)} - h_{\ell_i-\frac{1}{2}}^{(k+1)} \right) = 0, \qquad i = \dots, -1, 0, 1, \dots,
\end{equation}
where $\ell = 1, \dots, m$. The numerical flux is given by
\begin{equation*} \label{eq:Newton_flux_system}
    h_{\ell_i+\frac{1}{2}}^{(k+1)} = \sum_{\kappa=1}^m \left( \hat{f}^{(\pm)}_{\ell,\phi_\kappa}(\iter[\kappa_i]{v}{k}, \iter[\kappa_{i+1}]{v}{k}) \Delta v_{\kappa_{i+1}} \pm \hat{f}^{(\pm)}_{\ell,\phi_\kappa}(\iter[\kappa_{i+1}]{v}{k}, \iter[\kappa_i]{v}{k}) \Delta v_{\kappa_i} \right) + \hat{f}^{(\pm)}_\ell(\iter[\ell_{i+1}]{v}{k}, \iter[\ell_i]{v}{k}).
\end{equation*}
Further, if $\xvec{g}'(\vec{u} \kron \xscal{1})$ is nonsingular for any vector $\vec{u}$ with non-zero elements, then by the same argument as in the scalar case, the numerical flux is consistent. We may thus conclude:

% THEOREM: SYSTEM NEWTON
\begin{theorem} \label{thm:system_Newton}
Choose $\itersc[\ell_i]{v}{0} = u_{\ell_i}^n$ for each $i$ and $\ell = 1, \dots, m$ and suppose that $\xvec{g}'(\vec{u} \kron \xscal{1})$ is nonsingular for any vector $\vec{u}$ with non-zero elements. Then Newton's method, applied to the discretization \eqref{eq:Newton_vector_system}, is locally conservative and flux consistent.
\end{theorem}
% END OF THEOREM

Motivated by the conservation and flux consistency of Newton's method we conjecture that under conditions similar to \ref{assumptions}, the numerical approximation converges to a weak solution of the original conservation law \eqref{eq:conservation_law}. In \ref{sec:numerical_experiments} we present numerical results supporting this conjecture.

The solution $\Delta \vec{v}$ to the linear systems arising within Newton's method will in practice be approximated using a second iterative method. If we for this purpose consider pseudo-time iterations, then by the conservation and flux consistency of Newton's method, \ref{thm:local_conservation} applies. The only change necessary is to choose the initial guess $\Delta \itervec{u}{0}$ as the null-vector. We summarize these observations in the following:

% THEOREM: INEXACT NEWTON PSEUDO
\begin{theorem} \label{thm:inexact_Newton_pseudo}
Apply $K$ iterations of Newton's method to the implicit scheme \eqref{eq:FV} with initial guess $\itervec{v}{0} = \xvec{u}^n$. Approximate the solution to the linear system in the $j$th Newton iteration using $N_j$ pseudo-time iterations with ERK methods and initial guess $\Delta \itervec{u}{0} = \xvec{0}$. The Newton-pseudo-time iterations are locally conservative with numerical fluxes consistent with $c \vec{f}$, where
\begin{equation} \label{eq:wave_speed_Newton}
    c = 1 - \prod_{l=0}^{\sum_j^K N_j-1} \phi_l(-\mu_{l}).
\end{equation}
\end{theorem}
% END OF THEOREM

% PROOF
\begin{proof}
Since Newton's method is locally conservative and flux consistent, the result is a direct consequence of \ref{thm:local_conservation}.
\end{proof}
% END OF PROOF

%========================================================================

\section{Krylov subspace methods} \label{sec:krylov}

Krylov subspace methods are a more common choice than pseudo-time iterations for approximating the solutions of the linear systems arising within Newton's method. The goal of this section is to establish local conservation of Krylov subspace methods, and thereby also of Newton-Krylov methods, by relating them to pseudo-time iterations.

Let us consider a system of linear conservation laws and a correponding linearly implicit discretization in the form \eqref{eq:FV}. Due to the linearity of the numerical flux, there is a matrix $\xvec{D}$ such that the scheme can be expressed as
$$
\underbrace{\left( \frac{\xvec{I}}{\dt} + \frac{\xvec{D}}{\dx} \right)}_{\xvec{M}} \xvec{u}^{n+1} = \underbrace{ \frac{\xvec{u}^n}{\dt} }_{\xvec{d}}.
$$
We assume that $\xvec{M}$ is an invertible matrix.

By definition, the $(k+1)$st iteration of a Krylov subspace method finds an approximate solution $\itervec{w}{k+1}$ to $\xvec{M} \xvec{u} = \xvec{d}$ such that $\itervec{w}{k+1} - \itervec{w}{0} \in \mathcal{K}_{k+1}(\xvec{M}, \itervec{r}{0})$. The Krylov subspace $\mathcal{K}_{k+1}(\xvec{M}, \itervec{r}{0})$ is defined as
\begin{equation} \label{eq:Krylov_subspace}
\mathcal{K}_{k+1}(\xvec{M}, \itervec{r}{0}) := \vspan{ \{ \itervec{r}{0}, \xvec{M} \itervec{r}{0}, \dots, \xvec{M}^k \itervec{r}{0} \} },
\end{equation}
and $\itervec{r}{0} = \xvec{d} - \xvec{M} \itervec{w}{0}$ is the initial residual. In other words, there are coefficients $\alpha_0, \dots, \alpha_k$ such that
\begin{equation} \label{eq:Krylov_solution}
\itervec{w}{k+1} = \itervec{w}{0} + \left( \sum_{\ell=0}^k \alpha_\ell \xvec{M}^\ell \right) \itervec{r}{0}.
\end{equation}

We will now show that all Krylov subspace methods are locally conservative. To this end, we first establish a connection between Krylov subspace methods and pseudo-time iterations with explicit RK methods:

% LEMMA: KRYLOV RK EQUIVALENCE
\begin{lemma} \label{lemma:Krylov_RK_equivalence}
Let $\itervec{w}{k+1}$ denote the $(k+1)$st iterate of a Krylov subspace method. There exists an ERK method with $s=k+1$ stages such that the first iterate $\itervec{u}{1}$ of a pseudo-time iteration using this method satisfies $\itervec{u}{1} = \itervec{w}{k+1}$. 
\end{lemma}
% END OF LEMMA

% PROOF
\begin{proof}
We prove the lemma in two steps: The first step is to show that
$$
\itervec{u}{1} - \itervec{u}{0} \in \mathcal{K}_{k+1}(\xvec{M},\itervec{r}{0})
$$
when $\itervec{u}{0} = \itervec{w}{0}$. In other words, a single pseudo-time iteration with a $(k+1)$-stage explicit RK methods yields an approximation from the appropriate Krylov subspace. The second step is to choose $(\tmat{A}, \tmat{b}, \tmat{c})$ for the Runge-Kutta method such that the coefficients $\alpha_0, \dots, \alpha_k$ in \eqref{eq:Krylov_solution} are recovered.

For the first step, recall that pseudo-time iterations approximate solutions to the steady state problem
$$
\xvec{u}_\tau = \xvec{d} - \xvec{M} \xvec{u} =: \xvec{r}(\xvec{u}), \qquad \xvec{u}(0) = \itervec{u}{0}.
$$
A single iteration with an explicit RK method is obtained as
\begin{equation} \label{eq:RK_step_system}
\begin{aligned}
\tvec{\xvec{U}} &= \tvec{1} \kron \itervec{u}{0} + \dtau (\tvec{A} \kron \xvec{I}) \tvec{\xvec{R}}(\tvec{\xvec{U}}), \\
\itervec{u}{1} &= \itervec{u}{0} + \dtau (\tvec{b}^\top \kron \xvec{I}) \tvec{\xvec{R}} (\tvec{\xvec{U}}),
\end{aligned}
\end{equation}
where
\begin{equation*}
\tvec{\xvec{R}}(\tvec{\xvec{U}}) = ( (\xvec{d} - \xvec{M} \xvec{U}_1)^\top, \dots, (\xvec{d} - \xvec{M} \xvec{U}_s)^\top )^\top 
= (\tvec{1} \kron \xvec{d}) - (\tvec{I} \kron \xvec{M}) \tvec{\xvec{U}}.
\end{equation*}
Inserting $\tvec{\xvec{U}}$ from \eqref{eq:RK_step_system} into the latter expression leads to
\begin{align*}
\tvec{\xvec{R}}(\tvec{\xvec{U}}) &= (\tvec{1} \kron \xvec{d}) - (\tvec{I} \kron \xvec{M}) \left[ \tvec{1} \kron \itervec{u}{0} + \dtau (\tvec{A} \kron \xvec{I}) \tvec{\xvec{R}}(\tvec{\xvec{U}}) \right] \\
&= (\tvec{1} \kron \xvec{d}) - (\tvec{1} \kron \xvec{M} \itervec{u}{0}) - \dtau (\tvec{A} \kron \xvec{M}) \tvec{\xvec{R}}(\tvec{\xvec{U}}) \\
&= (\tvec{1} \kron \itervec{r}{0}) - \dtau (\tvec{A} \kron \xvec{M}) \tvec{\xvec{R}}(\tvec{\xvec{U}}).
\end{align*}
Solving for $\tvec{\xvec{R}}(\tvec{\xvec{U}})$ results in
$$
\tvec{\xvec{R}}(\tvec{\xvec{U}}) = \left[ \tvec{I} \kron \xvec{I} + \dtau (\tvec{A} \kron \xvec{M}) \right]^{-1} (\tvec{1} \kron \itervec{r}{0}).
$$
Inserting this expression into $\itervec{u}{1}$ from \eqref{eq:RK_step_system} yields
\begin{align*}
\itervec{u}{1} &= \itervec{u}{0} + \dtau (\tvec{b} \kron \xvec{I}) \left[ \tvec{I} \kron \xvec{I} + \dtau (\tvec{A} \kron \xvec{M}) \right]^{-1} (\tvec{1} \kron \itervec{r}{0}) \\
&= \itervec{u}{0} + \xvec{M}^{-1} \left( \dtau (\tvec{b} \kron \xvec{M}) \left[ \tvec{I} \kron \xvec{I} + \dtau (\tvec{A} \kron \xvec{M}) \right]^{-1} (\tvec{1} \kron \xvec{I}) \right) \itervec{r}{0}.
\end{align*}
We identify the paranthesised portion of this expression as $\xvec{I} - \phi(-\dtau \xvec{M})$, where $\phi(z)$ is the stability function of the RK method. Thus,
\begin{equation} \label{eq:RK_Krylov}
\itervec{u}{1} = \itervec{u}{0} + \xvec{M}^{-1} (\xvec{I} - \phi(-\dtau \xvec{M})) \itervec{r}{0}.
\end{equation}
Now, $\phi(z)$ is a polynomial of degree $s=k+1$ and from \eqref{eq:stability_function} we see that its constant term is $1$. Hence, $\xvec{M}^{-1}(\xvec{I} - \phi(-\xvec{M}))$ is a polynomial in $\xvec{M}$ of degree $k$. Consequently, $\itervec{u}{1} - \itervec{u}{0} \in \mathcal{K}_{k+1}(\xvec{M},\itervec{r}{0})$, which completes the first step of the proof.

For the second step we must find an explicit RK method whose stability polynomial satisfies $z^{-1}(1-\phi(-z)) = \sum_{\ell=0}^k \alpha_\ell z^\ell$. We do not require this RK method to be accurate or even consistent; it suffices that it exists.

Consider the $(k+1)$-stage explicit RK method defined by the Butcher tableau

\[
\renewcommand\arraystretch{1.2}
\begin{array}
{c|ccccc}
0 & 0 \\
c_1 & -a & 0 \\
c_2 & & -a & 0 \\
\vdots & & & \ddots & \ddots \\
c_k & & & & -a & 0 \\
\hline
& b_0 & b_1 & \dots & b_{k-1} & b_k 
\end{array}
\]

\noindent By direct computation we can verify that $(\tmat{I} - z \tmat{A})^{-1}$ is a lower triangular Toeplitz matrix with elements $(-az)^\ell$ on the $\ell$th subdiagonal ($\ell=0$ being the main diagonal and $\ell = k$ being the bottom left element). From \eqref{eq:stability_function}, the stability function of this method is therefore
$$
\phi(z) = 1 + z \sum_{i=0}^k \sum_{j=i}^k b_j (-az)^i.
$$
Consequently we are looking for coefficients $b_j$ that satisfy the relation
$$
\frac{1 - \phi(-z)}{z} = \sum_{i=0}^k \sum_{j=i}^k b_j (az)^i = \sum_{\ell=0}^k \alpha_\ell z^\ell.
$$
Matching exponents of $z$, this amounts to solving the triangular system
$$
\begin{bmatrix}
1 & 1 & \dots & 1 & 1 \\
0 & 1 & \dots & 1 & 1 \\
& & \ddots & & \\
0 & 0 & \dots & 0 & 1
\end{bmatrix}
\begin{pmatrix}
b_0 \\ b_1 \\ \vdots \\ b_k
\end{pmatrix}
=
\begin{pmatrix}
a^{-0} \alpha_0 \\
a^{-1} \alpha_1 \\
\vdots \\
a^{-k} \alpha_k
\end{pmatrix}.
$$
This system is uniquely solvable for any coefficients $\alpha_\ell, \, \ell = 0, \dots, k$ and $a \neq 0$. Hence, an explicit RK method can always be found such that a single pseudo-time iteration reproduces the $(k+1)$th Krylov vector when the initial guess is given by $\itervec{u}{0} = \itervec{w}{0}$.
\end{proof}
% END OF PROOF

With \ref{lemma:Krylov_RK_equivalence} established, the following is a direct consequence of \ref{thm:local_conservation}:

% THEOREM: KRYLOV_LOCALLY_CONSERVATIVE
\begin{theorem} \label{thm:Krylov_locally_conservative}
With the initial guess $\itervec{w}{0} = \xvec{u}^n$, Krylov subspace methods applied to the locally conservative linear discretization \eqref{eq:FV} are locally conservative.
\end{theorem}
% END OF THEOREM

% PROOF
\begin{proof}
For any iteration of any Krylov subspace method, \ref{lemma:Krylov_RK_equivalence} shows that there is an equivalent explicit RK method that yields the same numerical solution in one pseudo-time iteration. By \ref{thm:local_conservation}, all explicit RK methods are locally conservative, hence so are Krylov subspace methods.
\end{proof}
% END OF PROOF

We do not know what the coefficients $\alpha_0, \dots, \alpha_k$ are and consequently we cannot identify a corresponding RK method. Further, each $\alpha_j$ will in general change with every iteration. We can therefore not know a priori if a given Krylov subspace method is flux consistent. We can also not apply the extension of the Lax-Wendroff theorem in \ref{thm:LW} since $\alpha_j$ in general will not remain constant upon grid refinement.

On the other hand, \ref{thm:Krylov_locally_conservative} suggests that we can apply one pseudo-time iteration with the explicit Euler method prior to applying the Krylov subspace method in order to enforce consistency, without violating local conservation. Further yet, since \ref{thm:local_conservation} is indifferent to the RK method used in previous iterations, we can conclude that \ref{thm:Krylov_locally_conservative} also applies to \emph{restarted} Krylov methods, with each restart corresponding to a pseudo-time iteration with its own RK method. In fact, local conservation will be retained even if we swap Krylov method mid-solve, or if we mix Krylov subspace methods and pseudo-time iterations.

Finally we note that \ref{thm:Krylov_locally_conservative} together with \ref{thm:system_Newton} implies that Newton-Krylov methods are locally conservative:

% THEOREM: INEXACT NEWTON KRYLOV
\begin{theorem} \label{thm:inexact_Newton_Krylov}
Apply Newton's method to the implicit scheme \eqref{eq:FV} with initial guess $\itervec{v}{0} = \xvec{u}^n$. Approximate the solution to the each linear system within the Newton iterations using any Krylov subspace method with initial guess $\Delta \itervec{w}{0} = \xvec{0}$. Then the Newton-Krylov iterations are locally conservative.
\end{theorem}
% END OF THEOREM

% PROOF
\begin{proof}
Since Newton's method is locally conservative and flux consistent by \ref{thm:system_Newton}, the result is a direct consequence of \ref{thm:Krylov_locally_conservative}.
\end{proof}
% END OF PROOF

%========================================================================

\section{Numerical experiments} \label{sec:numerical_experiments}

As target problem for the experiments in this section we use the 2D compressible Euler equations,
\begin{equation} \label{eq:Eulers_equations}
\begin{bmatrix}
\rho \\ \rho u \\ \rho v \\ \rho E
\end{bmatrix}_t
+
\begin{bmatrix}
\rho u \\
\rho u^2 + p \\
\rho u v \\
(\rho E + p) u
\end{bmatrix}_x
+
\begin{bmatrix}
\rho v \\
\rho u v \\
\rho v^2 + p \\
(\rho E + p) v
\end{bmatrix}_y
= 0,
\end{equation}
posed on the domain $(x,y) \in (-5, 15] \times (-5, 5]$. Here, $\rho, u, v, E$ and $p$ respectively denote density, horizontal and vertical velocity components, total energy per unit mass and pressure. The pressure is related to the other variables through the equation of state
$$
p = (\gamma-1) \rho e,
$$
where $\gamma = 1.4$ and $e = E - (u^2 + v^2)/2$ is the internal energy density. The domain is taken to be periodic in both spatial coordinates. The setting is the isentropic vortex problem \cite{shu1998essentially} with initial conditions
\begin{align*}
\rho_0 &= \left( \frac{1 - \epsilon^2 (\gamma - 1) M_\infty^2}{8 \pi^2} \exp{(r)} \right)^\frac{1}{\gamma - 1}, \\
u_0 &= 1 - \frac{\epsilon y}{2 \pi} \exp{(r/2)}, \\
v_0 &= \frac{\epsilon x}{2 \pi} \exp{(r/2)}, \\
p_0 &= \frac{\rho_0^\gamma}{\gamma M_\infty^2},
\end{align*}
where $r = 1 - x^2 - y^2$. Here, $\epsilon = 5$ is the circulation and $M_\infty = 0.5$ is the Mach number. As the solution evolves in time, the initial vortex propagates in the horizontal direction with unit speed.

This problem does not adhere to the assumptions made throughout the paper since it is posed in 2D. Nevertheless, we will see shortly that the theoretical results presented so far appear to hold also in this setting.

%========================================================================

\subsection{Convergence tests}

\subsubsection*{Pseudo-time iterations: SSPRK3}

We begin by corroborating the extension of the Lax-Wendroff theorem presented in \ref{thm:LW}. With implicit Euler used in time, experimental verification of this theorem was provided in \cite{birken2021conservative}. However, \ref{thm:Runge-Kutta} implies that the result holds also when other implicit RK methods are used in time, so long as they fulfill condition \eqref{eq:v_conditions}.

In the following experiments we discretize the Euler equations in space using finite volumes with central fluxes and the three-stage Lobatto IIIC method in time. This fully implicit RK method satisfies \eqref{eq:v_conditions} with the vector $\ivec{v} = (0,0,1)^\top$. The simulations are run to time $t=0.1$ with space and time grids satisfying by $\dy = \dx/2$ and $\dt = \dx/4$. In each time step, the discrete solution is approximated using three pseudo-time iterations using the explicit strong stability preserving method SSPRK3 from \cite{shu1988efficient}. The pseudo-time steps are chosen to be $\dtau_i = \dt/\sqrt{i}, \, i = 1,2,3$. The stability function of SSPRK3 is given by
$$
\phi(z) = 1 + z + z^2/2 + z^3/6.
$$
\ref{thm:LW} therefore predicts that the solution of the scheme, if convergent, approaches the solution to a modified conservation law with modification constant
$$
c(\mu_0,\mu_1,\mu_2) = 1 - \phi(\mu_0)\phi(\mu_1)\phi(\mu_2) \approx 0.9101.
$$
For this particular problem, the presence of $c$ causes a corresponding reduction of the vortex propagation speed. \ref{fig:pseudo} shows the $L^2$-error in the density component of the solution, as measured with respect to the exact solution of the original system of conservation laws \eqref{eq:Eulers_equations} and that of the modified version. Clearly the numerical solution approaches the solution of the modified equations, not the original ones.

Shown also is the error, measured with respect to the original conservation law, of the same scheme but where flux consistency has been enforced using a single pseudo-time iteration with explicit Euler and $\mu=1$. As expected, this ensures convergence towards the correct solution.

\subsubsection*{Newton-SSPRK3}

We now repeat the experiment but add two Newton iterations per time step. The solutions to the linear systems are then approximated using the same three pseudo-time iterations as before. In addition to the discretization described previously, we also perform the experiment with a different one: In space, the entropy conservative and kinetic energy preserving finite volume scheme of Chandrashekar \cite{chandrashekar2013kinetic} is used and in time, the two-stage Radau IIA method, which satisfies \eqref{eq:v_conditions} with the vector $\ivec{v} = (0,1)^\top$.

\ref{thm:inexact_Newton_pseudo} suggests that the resulting schemes are consistent with a modified system of conservation laws with modification constant
$$
c = 1 - (\phi(\mu_0)\phi(\mu_1)\phi(\mu_2))^2 \approx 0.9919.
$$
Rather than computing the Jacobian explicitly, it is approximated using finite differences such that for any two vectors $\xvec{u}$ and $\xvec{v}$,
\begin{equation} \label{eq:inexact_Jacobian}
\xvec{g}'(\xvec{u}) \xvec{v} \approx \frac{\xvec{g}(\xvec{u} + \epsilon \xvec{v}) - \xvec{g}(\xvec{u})}{\epsilon}, \quad \epsilon = \frac{10^{-7}}{\| \xvec{v} \|}.
\end{equation}
Although an extension of the Lax-Wendroff theorem incorporating Newton's method is not yet available, \ref{fig:Newton-pseudo} indicates convergence towards the solution of the modified equations, not of the original ones. Once again, this is remedied by enforcing flux consistency with a single explicit Euler iteration. Here, the errors corresponding to the two different space-time discretizations are similar enough that the lines are on top of each other. 

\begin{figure}[tbhp]
\centering
\subfloat[SSPRK3]{\label{fig:pseudo}\includegraphics[width=0.5\textwidth]{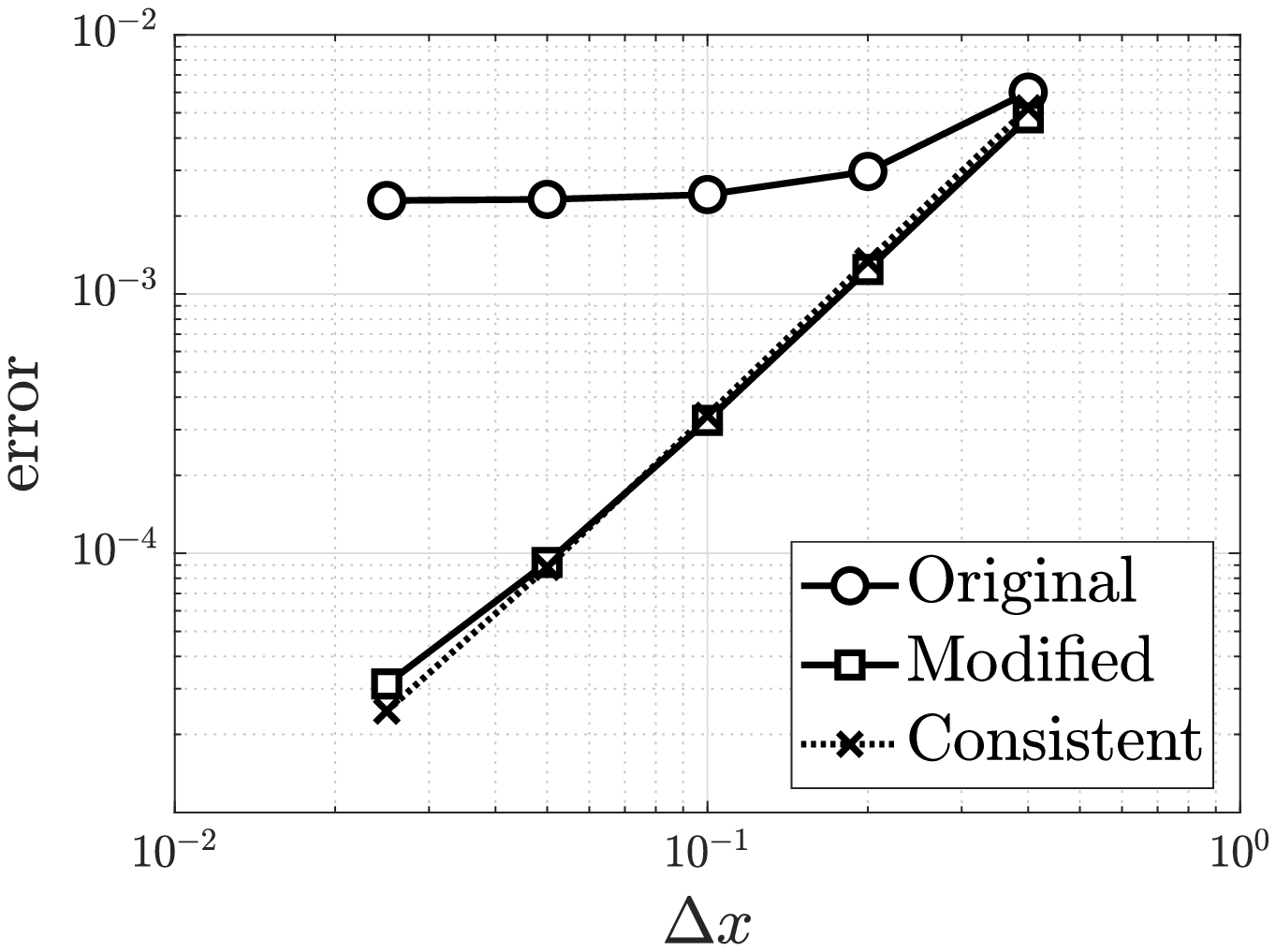}}
\subfloat[Newton-SSPRK3]{\label{fig:Newton-pseudo}\includegraphics[width=0.5\textwidth]{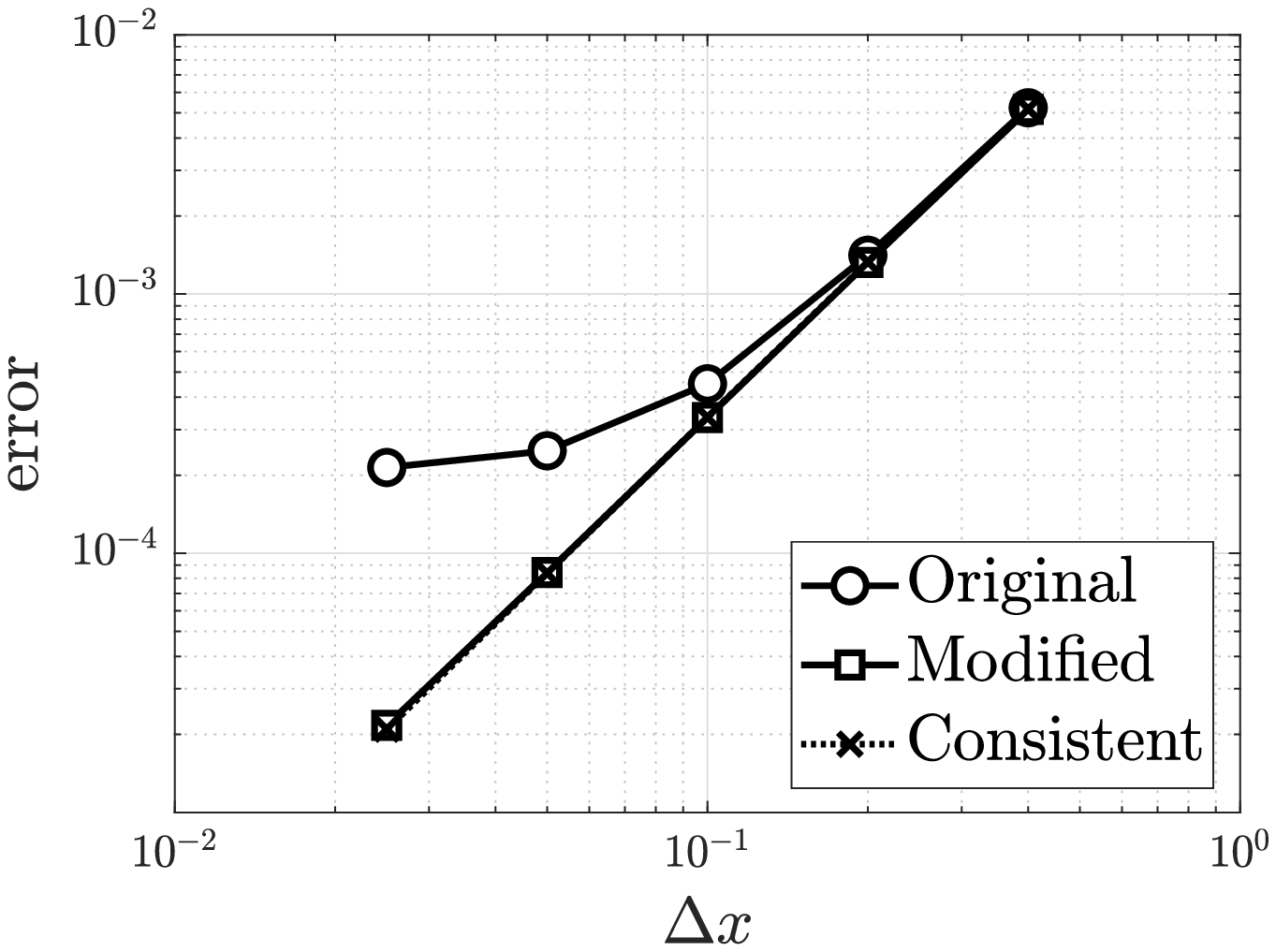}}
\caption{Density error upon grid refinement for the compressible Euler equations. Convergence is seen towards the modified conservation laws (Modified), not the original ones (Original), unless flux consistency is enforced (Consistent). (a) Pseudo-time iterations using SSPRK3. (b) Newton's method with SSPRK3 as subsolver for the linear systems.}
\label{fig:convergence_pseudo}
\end{figure}

\subsection*{Newton-Krylov: Fixed iterations}

Newton-Krylov methods are a more common choice of solver than Newton with pseudo-time iterations. \ref{thm:inexact_Newton_Krylov} establishes that they are all locally conservative. Conceivably, a Lax-Wendroff type result might be available for these methods, although we have not presented one here. It is of interest to explore experimentally if convergence is observed, and if so, towards what solution.

We repeat the previous experiment, this time with the three-stage Lobatto IIIC method in time and Chandrashekar's finite volume scheme in space, but replace SSPRK3 with GMRES. Here we consider four cases: Using one Newton and one Krylov iteration per time step (N1K1); one Newton and two Krylov iterations (N1K2); one Newton iteration with GMRES run to a tolerance of $10^{-14}$ (N1), i.e. effectively with 'exact' linear solves; Newton-GMRES run until the Newton residual is smaller than $10^{-15}$ (Exact), i.e. effectively with 'exact' nonlinear solves. \ref{fig:fixed_iter} shows the error with respect to the original conservation law of the four schemes. The errors for N1K2, N1 and the exact solver are on top of each other, suggesting that the discretization dominates the error. The numerical solutions appear to converge to the correct solution. However, N1K1 displays a different behaviour, suggesting that the iteration error dominates. The error curve appears to flatten as the grid is refined, although it cannot be deduced whether it will continue towards zero or if it reaches a plateau. Thus, it remains unclear if GMRES is flux consistent.

\subsection*{Newton-Krylov: Tolerance}

In practice, the number of Newton iterations will not be preset but rather governed by a relative and/or an absolute tolerance. Here, we set both of these to a value $tol$ and once again explore the convergence behaviour. To set the tolerances of the GMRES iterations we follow the procedure described by Eisenstat and Walker \cite{eisenstat1996choosing} with parameters $\gamma = \eta_{max} = 0.9$; see \cite[Chapter 6]{kelley1995iterative} for details. The density errors using $tol \in \{ 10^{-3}, 10^{-4}, 10^{-5} \}$ are shown in \ref{fig:tolerance}.

Two distinct phenomena can be observed: Firstly, when $tol = 10^{-5}$, the convergence behaviour changes from one similar to the exact solver in \ref{fig:fixed_iter} to one resembling N1K1. Thus, the error is seen to change from being discretization dominated to being iteration dominated as the grid is refined. With $tol = 10^{-4}$ the iteration error appears to dominate throughout.

Secondly, when $tol = 10^{-3}$ the error eventually stops converging. This behavior is explained by the observation that with a fixed tolerance, the initial guess will be a sufficiently accurate approximation of the solution if $\dt$ is small enough. In that case, the Newton and GMRES iterations are terminated without updating the solution. 
%We expect this to eventually happen with the other tolerances as well. Thus, a tolerance governed scheme can only converge if the tolerance is reduced upon grid refinement.

\begin{figure}[tbhp]
\centering
\subfloat[Fixed iterations]{\label{fig:fixed_iter}\includegraphics[width=0.5\textwidth]{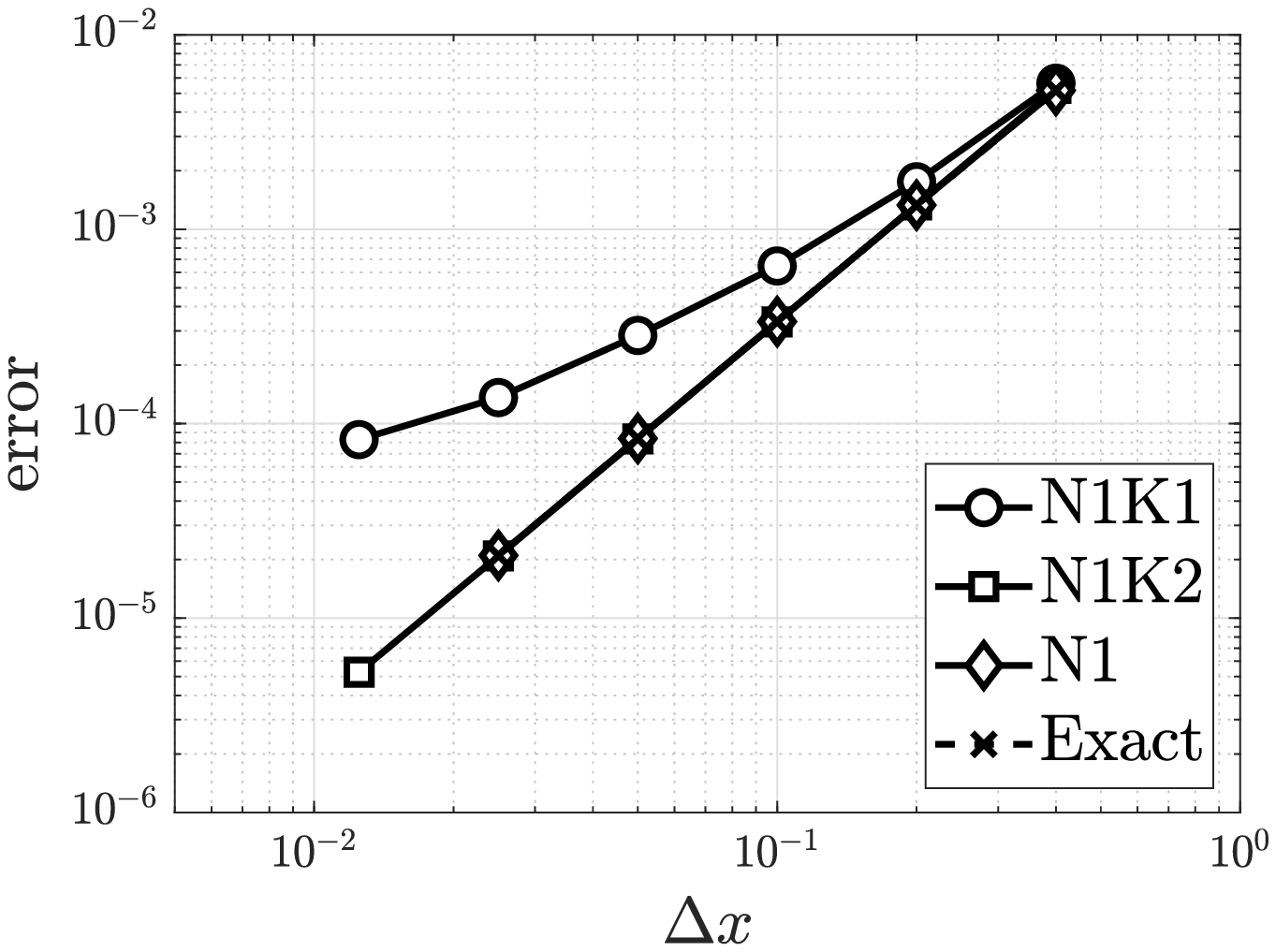}}
\subfloat[Tolerance]{\label{fig:tolerance}\includegraphics[width=0.5\textwidth]{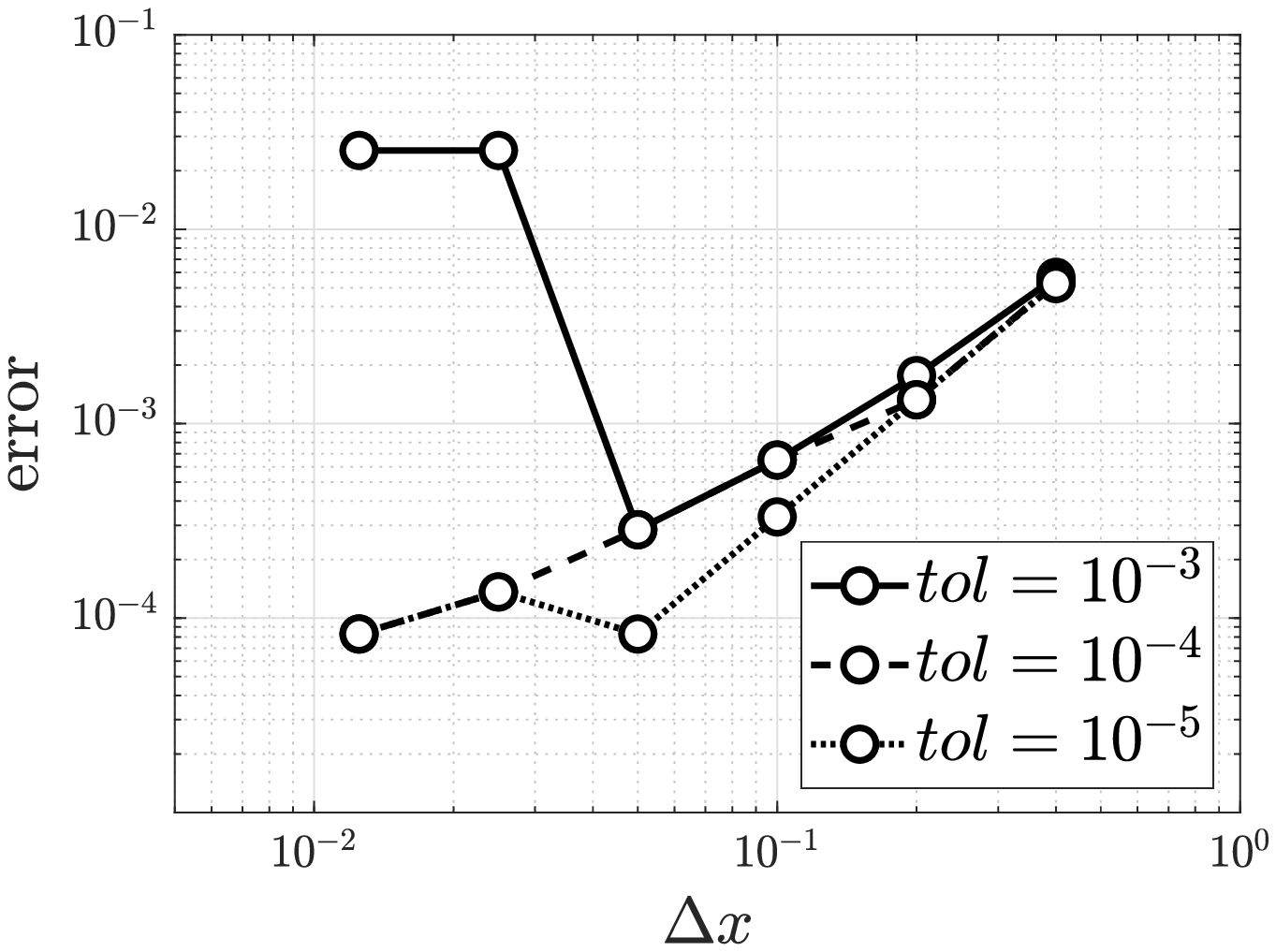}}
\caption{Density error upon grid refinement for the compressible Euler equations. (a) The number of iterations is fixed: One Newton and one GMRES iteration per time step (N1K1); One Newton and two GMRES iterations (N1K2); one Newton with a nearly exact linear solver (N1); nearly exact Newton-GMRES (Exact). (b) Tolerance governed Newton-GMRES with the Eisenstat-Walker procedure.}
\label{fig:convergence_Newton-GMRES}
\end{figure}

%========================================================================

\subsection{Acceleration experiments}

In \cite{birken2021conservative}, numerical experiments were performed indicating that considerable efficiency gains can be made with pseudo-time iterations by enforcing flux consistency. It is not clear whether Newton-Krylov methods are flux consistent. It is therefore worthwhile exploring if enforced flux consistency leads to efficiency gains also in this case. We use the tolerance governed Newton-GMRES solver with the Eisenstat-Walker procedure and compare the standard solver with one where flux consistency is enforced with explicit Euler before each call to GMRES.

In this experiment we use Chandrashekar's finite volume scheme with the implicit Euler method in time to compute a single time step with $\dt = 0.1$. The running cost of the two solvers is measured in terms of the number of evaluations of the full space discretization. Each Newton iteration requires a single such evaluation; see \eqref{eq:Newtons_method}. Each GMRES iteration also needs a single function evaluation during the computation of the approximate Jacobian matrix-vector product \eqref{eq:inexact_Jacobian}.

The $k$th pseudo-time step with explicit Euler applied within the $j$th Newton iteration takes the form
$$
\frac{\Delta \itervec{u}{k+1} - \Delta \itervec{u}{k}}{\dtau_k} + \xvec{g}'(\itervec{v}{j}) \Delta \itervec{u}{k} + \xvec{g}(\itervec{v}{j}) = \xvec{0}.
$$
Consider the case $k=0$ and recall from \ref{thm:inexact_Newton_pseudo} that local conservation follows if the initial guess is $\Delta \itervec{u}{0} = \xvec{0}$. Thus, enforcing flux consistency by one pseudo-time iteration with explicit Euler simply amounts to setting $\Delta \itervec{u}{1} = -\dt \xvec{g}(\itervec{v}{j})$. This single function evaluation is already computed within Newton's method, hence flux consistency comes at no additional cost. In fact, the only change necessary to the solver is to alter the initial guess for GMRES from $\Delta \itervec{w}{0} = \xvec{0}$ to $\Delta \itervec{w}{0} = -\dt \xvec{g}(\itervec{v}{j})$.

% REMARK: DELTA t VS mu
%\begin{remark}
%The consistent initial guess used here assumes that $\xvec{g}$ has the particular form given in \eqref{eq:Newtons_method}. A common alternative is to instead pass $\tilde{\xvec{g}} = \dt \xvec{g}$ to the nonlinear solver. With this scaling, the appropriate initial guess is instead $\itervec{w}{0} = -\tilde{\xvec{g}}(\itervec{v}{j})$.
%\end{remark}
% END OF REMARK

Three cases with different CFL numbers are considered. \ref{fig:Eval_Res} shows the number of function evaluations required to reach a particular residual $\| \xvec{g}(\itervec{v}{j}) \|$, where the $L^2$-norm is used. \ref{fig:Res_Eval_Newton} shows the function evaluations distributed accross the Newton iterations. In all cases, the flux consistent initial guess (dotted lines) reduces the number of necessary iterations for residuals greater than roughly $10^{-3}$, compared to the regular solver (solid lines). However, for smaller residuals the situation varies with the CFL number. At the smallest CFL, flux consistency remains beneficial even for finer tolerances. However, for the largest CFL the opposite trend is seen.

A possible explanation for these observations is that the explicit Euler method introduces significant errors to the numerical solution when the CFL number is large. Its stability region is small, hence problems with large CFL numbers are unsurprising. It is possible to find other explicit Runge-Kutta methods that enforce flux consistency, in principle with much larger stability regions. However, such methods will necessarily have more stages and thus impart additional costs on the solver. 

Finally we note that at the smallest tolerances, no discernable difference is seen between the two solvers. Presumably this happens because flux consistency is achieved, either exactly or very nearly, by the standard Newton-GMRES solver when the tolerance is small and the number of iterations is large. In conclusion, whether enforced flux consistency is beneficial for Newton-GMRES is case dependent.

\begin{figure}[tbhp]
\centering
\subfloat[]{\label{fig:Eval_Res}\includegraphics[width=0.5\textwidth]{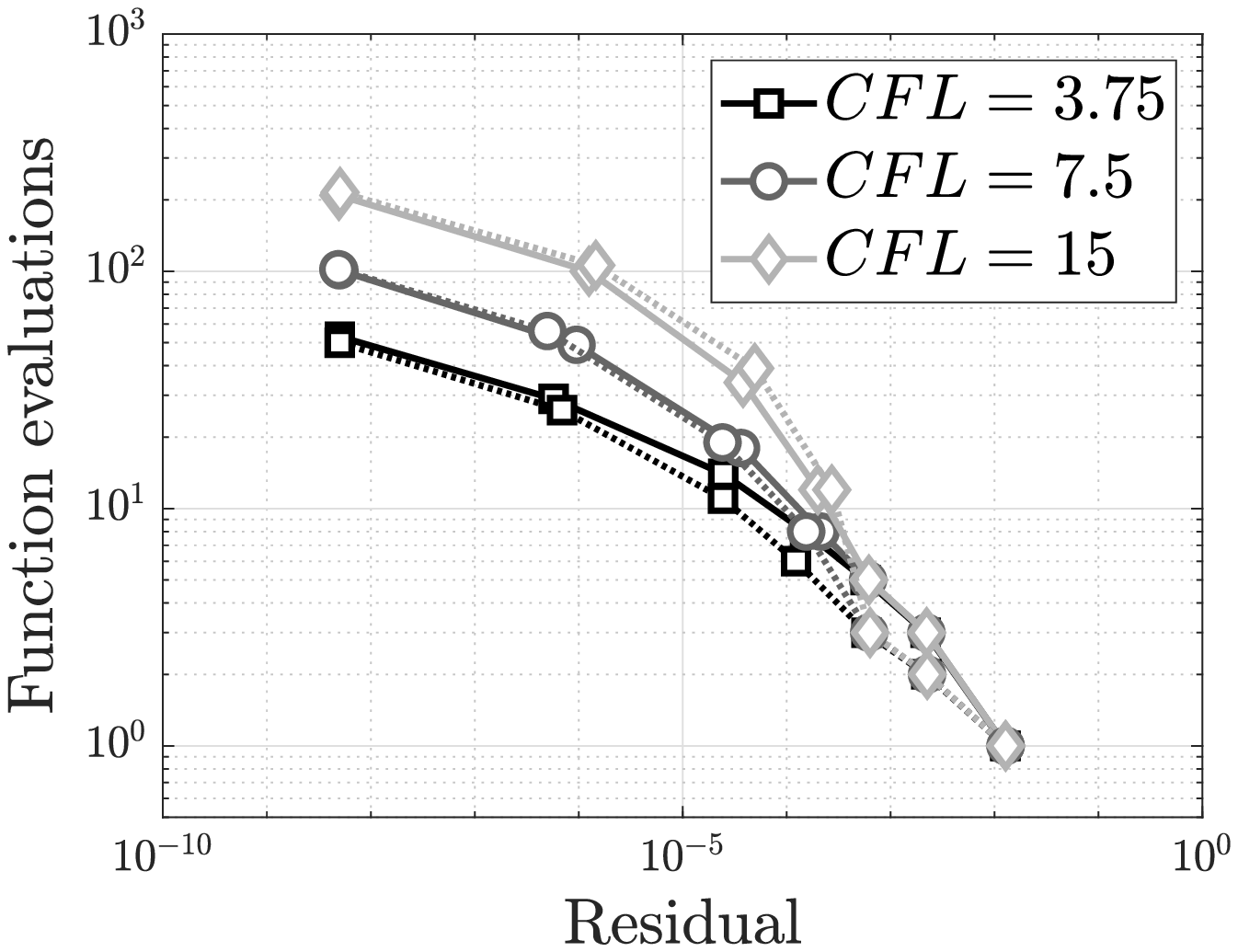}}
\subfloat[]{\label{fig:Res_Eval_Newton}\includegraphics[width=0.5\textwidth]{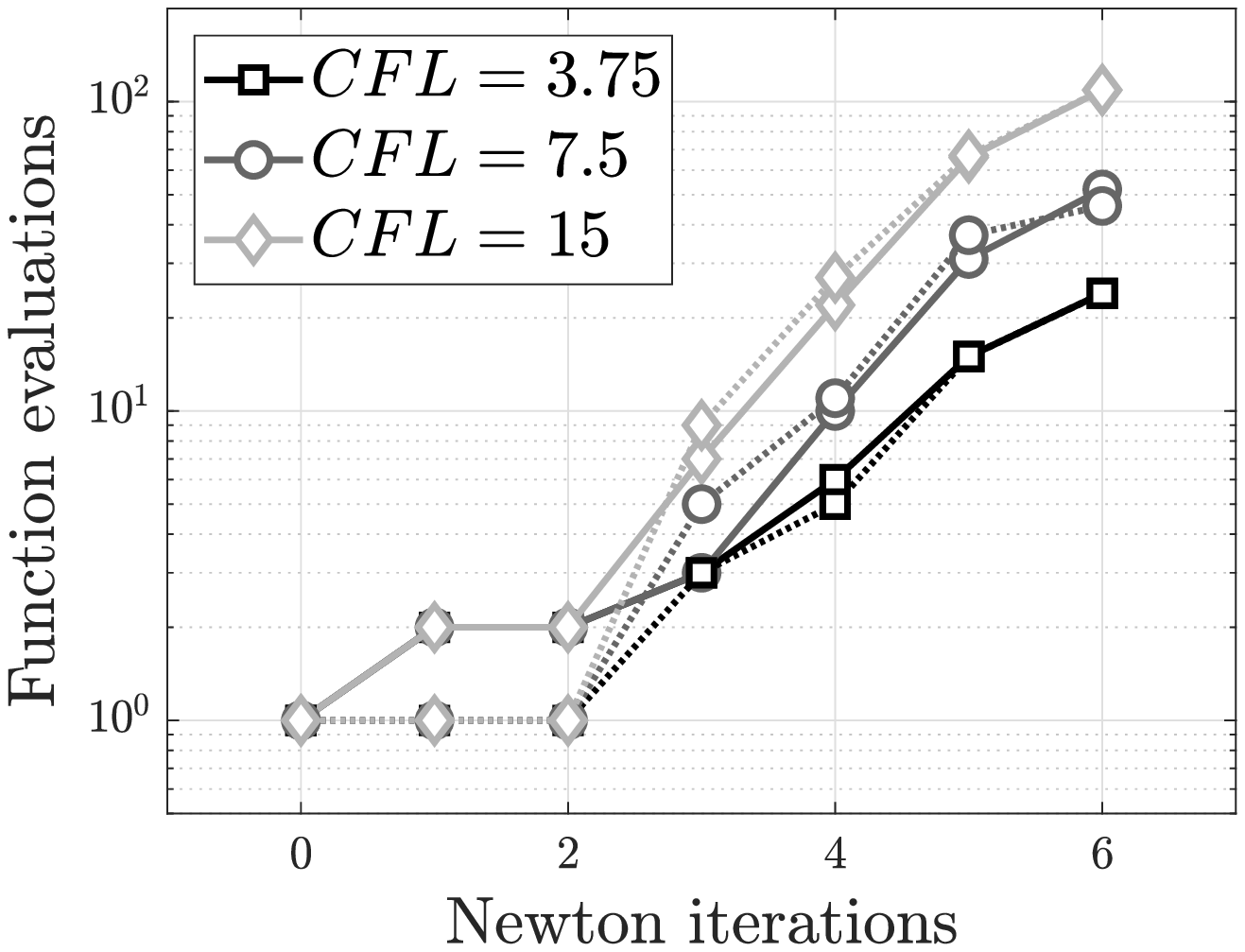}}
\caption{Efficiency study for Newton-GMRES with standard (solid lines) and flux consistent (dotted lines) initial guesses. (a) Function evaluations needed to reach a given residual. (b) Function evaluations per Newton iteration.}
\label{fig:residuals}
\end{figure}

%========================================================================

\section{Conclusions} \label{sec:conclusions}

In this paper, the concepts of locally conservative and flux consistent iterative methods have been introduced and shown to be of both theoretical and practical interest. Based on earlier work in \cite{birken2021conservative}, it was shown that pseudo-time iterations using explicit Runge-Kutta methods are locally conservative but not necessarily flux consistent, when applied to conservative discretizations of finite volume-type with a broad class of implicit Runge-Kutta methods. For 1D problems, an extension of the Lax-Wendroff theorem reveals convergence towards weak solutions of a temporally retarded system of conservation laws. Each equation is modified in the same way, namely by a particular scalar factor multiplying the spatial flux terms. Flux consistency, and thereby convergence, is recovered through a technique based on using the explicit Euler method.

Local conservation has further been established for all Krylov subspace methods, with and without restarts, as well as for Newton's method under the assumption of bivariate fluxes. Thus it follows that Newton-Krylov methods are locally conservative, although not necessarily flux consistent. Numerical experiments with the 2D compressible Euler equations suggest that the role enforced flux consistency is case dependent. Its effect diminishes as the number of GMRES iterations grow, presumably because flux consistency is achieved automatically.

%========================================================================

\bibliography{references.bib}

\begin{thebibliography}{10}

\bibitem{barth1987analysis}
{\sc T.~Barth}, {\em Analysis of implicit local linearization techniques for
  upwind and {TVD} algorithms}, in 25th AIAA Aerospace Sciences Meeting, 1987,
  p.~595.

\bibitem{birken2021numerical}
{\sc P.~Birken}, {\em Numerical methods for unsteady compressible flow
  problems}, Chapman and Hall/CRC, 2021.

\bibitem{birken2019preconditioned}
{\sc P.~Birken, J.~Bull, and A.~Jameson}, {\em Preconditioned smoothers for the
  full approximation scheme for the {RANS} equations}, Journal of Scientific
  Computing, 78 (2019), pp.~995--1022.

\bibitem{birken2021conservative}
{\sc P.~Birken and V.~Linders}, {\em Conservative iterative methods for
  implicit discretizations of conservation laws}, arXiv preprint
  arXiv:2106.10088,  (2021).

\bibitem{boom2015high}
{\sc P.~D. Boom and D.~W. Zingg}, {\em High-order implicit time-marching
  methods based on generalized {S}ummation-{B}y-{P}arts operators}, SIAM
  Journal on Scientific Computing, 37 (2015), pp.~A2682--A2709.

\bibitem{chan2022efficient}
{\sc J.~Chan and C.~G. Taylor}, {\em {Efficient computation of Jacobian
  matrices for entropy stable summation-by-parts schemes}}, J. Comput. Phys.,
  448 (2022), p.~110701.

\bibitem{chandrashekar2013kinetic}
{\sc P.~Chandrashekar}, {\em {Kinetic energy preserving and entropy stable
  finite volume schemes for compressible Euler and Navier-Stokes equations}},
  Communications in Computational Physics, 14 (2013), pp.~1252--1286.

\bibitem{eisenstat1996choosing}
{\sc S.~C. Eisenstat and H.~F. Walker}, {\em {Choosing the forcing terms in an
  inexact Newton method}}, SIAM Journal on Scientific Computing, 17 (1996),
  pp.~16--32.

\bibitem{jespersen1983flux}
{\sc D.~Jespersen and T.~Pulliam}, {\em Flux vector splitting and approximate
  {N}ewton methods}, in 6th Computational Fluid Dynamics Conference Danvers,
  1983, p.~1899.

\bibitem{junqueira2014study}
{\sc C.~Junqueira-Junior, L.~C. Scalabrin, E.~Basso, and J.~L.~F. Azevedo},
  {\em Study of conservation on implicit techniques for unstructured finite
  volume {N}avier--{S}tokes solvers}, Journal of Aerospace Technology and
  Management, 6 (2014), pp.~267--280.

\bibitem{kelley1995iterative}
{\sc C.~T. Kelley}, {\em Iterative methods for linear and nonlinear equations},
  SIAM, 1995.

\bibitem{lax1959systems}
{\sc P.~Lax and B.~Wendroff}, {\em Systems of conservation laws}, tech. rep.,
  LOS ALAMOS NATIONAL LAB NM, 1959.

\bibitem{leveque1992numerical}
{\sc R.~J. LeVeque}, {\em Numerical methods for conservation laws}, vol.~3,
  Springer, 1992.

\bibitem{linders2020properties}
{\sc V.~Linders, J.~Nordstr{\"o}m, and S.~H. Frankel}, {\em Properties of
  {R}unge-{K}utta-{S}ummation-{B}y-{P}arts methods}, Journal of Computational
  Physics, 419 (2020), p.~109684.

\bibitem{nordstrom2013summation}
{\sc J.~Nordstr{\"o}m and T.~Lundquist}, {\em Summation-{B}y-{P}arts in time},
  Journal of Computational Physics, 251 (2013), pp.~487--499.

\bibitem{ranocha2019some}
{\sc H.~Ranocha}, {\em Some notes on {S}ummation {B}y {P}arts time integration
  methods}, Results in Applied Mathematics, 1 (2019), p.~100004.

\bibitem{shu1998essentially}
{\sc C.-W. Shu}, {\em Essentially non-oscillatory and weighted essentially
  non-oscillatory schemes for hyperbolic conservation laws}, in Advanced
  numerical approximation of nonlinear hyperbolic equations, Springer, 1998,
  pp.~325--432.

\bibitem{shu1988efficient}
{\sc C.-W. Shu and S.~Osher}, {\em Efficient implementation of essentially
  non-oscillatory shock-capturing schemes}, Journal of computational physics,
  77 (1988), pp.~439--471.

\bibitem{swanson2007convergence}
{\sc R.~C. Swanson, E.~Turkel, and C.-C. Rossow}, {\em Convergence acceleration
  of {R}unge--{K}utta schemes for solving the {N}avier--{S}tokes equations},
  Journal of Computational Physics, 224 (2007), pp.~365--388.

\bibitem{versbach2022theoretical}
{\sc L.~M. Versbach, V.~Linders, R.~Kl{\"o}fkorn, and P.~Birken}, {\em
  {Theoretical and practical aspects of space-time DG-SEM implementations}},
  arXiv preprint arXiv:2201.05800,  (2022).

\bibitem{wanner1996solving}
{\sc G.~Wanner and E.~Hairer}, {\em Solving ordinary differential equations
  {II}}, Springer Berlin Heidelberg, 1996.

\end{thebibliography}

%========================================================================

\end{document}